\documentclass[11pt]{amsart}
\usepackage{amsmath,amssymb,amsthm}
\usepackage{tikz}
\usepackage[letterpaper, margin=1in]{geometry}
\usepackage{todonotes}
\usetikzlibrary{decorations.markings,intersections}
\usepackage{verbatim}
\usepackage{enumitem}
\usepackage{graphicx}

\usepackage[backref]{hyperref}
\hypersetup{
    colorlinks,
    linkcolor={red!60!black},
    citecolor={green!60!black},
    urlcolor={blue!60!black},
    pdftitle={A unified half-integral Erdős-Pósa theorem for cycles in graphs labelled by multiple abelian groups},
    pdfauthor={J. Pascal Gollin and Kevin Hendrey and Ken-ichi Kawarabayashi and O-joung Kwon and Sang-il Oum}
}

\newcommand\abs[1]{\ensuremath{\lvert #1\rvert}}

\usepackage{thmtools, thm-restate}

\newtheorem{theorem}{Theorem}[section]
\newtheorem{lemma}[theorem]{Lemma}

\newtheorem{corollary}[theorem]{Corollary}
\newtheorem{proposition}[theorem]{Proposition}

\newtheorem{claim}{Claim}
\newtheorem*{claim*}{Claim}
\newtheorem{question}{Question}

\newtheorem{remark}[theorem]{Remark}

\usetikzlibrary{decorations.pathmorphing}
\usetikzlibrary{decorations.markings}

\theoremstyle{definition}

\newcommand\lcm{\operatorname{lcm}}

\newcommand\cP{\mathcal{P}}

\newcommand\cT{\mathcal{T}}

\newcommand{\gen}[1]{\ensuremath{\langle #1\rangle}}

\newcommand{\branch}{\ensuremath{V_{\neq 2}}}

\newenvironment{proofofclaim}[1][Proof.]{%
    \begin{proof}[{#1}]%
        }{%
    \end{proof}}

\begin{document}
\title[Half-integral Erd\H{o}s-P\'{o}sa for cycles in graphs labelled by multiple abelian groups]%
{A unified half-integral Erd\H{o}s-P\'{o}sa theorem for cycles in graphs labelled by multiple abelian groups} 

\author{J.~Pascal Gollin} 
\author{Kevin Hendrey}

\address[Gollin, Hendrey, Kwon, Oum]{\small Discrete Mathematics Group, Institute~for~Basic~Science~(IBS), Daejeon,~South~Korea.}

\author{Ken-ichi Kawarabayashi}
\address[Kawarabayashi]{National Institute of Informatics, 2-1-2, Hitotsubashi, Chiyoda-ku, Tokyo, Japan}

\author{O-joung Kwon}

\address[Kwon]{\small Department of Mathematics, Incheon~National~University, Incheon,~South~Korea.}

\author{Sang-il~Oum}
\address[Oum]{\small Department of Mathematical Sciences, KAIST,  Daejeon,~South~Korea.}

\email{pascalgollin@ibs.re.kr}
\email{kevinhendrey@ibs.re.kr}
\email{k{\textunderscore}keniti@nii.ac.jp}
\email{ojoungkwon@gmail.com}
\email{sangil@ibs.re.kr}
\thanks{All authors except the third author are supported by the Institute for Basic Science (IBS-R029-C1). The third author is supported by JSPS Kakenhi Grant Number JP18H05291. The fourth author is supported by the National Research Foundation of Korea (NRF) grant funded by the Ministry of Education (No. NRF-2018R1D1A1B07050294). }
\date{February 3rd, 2021}

\begin{abstract}
    Erd\H{o}s and P\'{o}sa proved in 1965 that there is a duality between the maximum size of a packing of cycles and the minimum size of a vertex set hitting all cycles.
    Such a duality does not hold if we restrict to odd cycles. 
    However, in 1999, Reed proved an analogue for odd cycles by relaxing packing to half-integral packing.
    We prove a far-reaching generalisation of the theorem of Reed; if the edges of a graph are labelled by finitely many abelian groups, then 
    there is a duality between the maximum size of a half-integral packing of cycles whose values avoid a fixed finite set for each abelian group
    and the minimum size of a vertex set hitting all such cycles. 
    
    A multitude of natural properties of cycles can be encoded in this setting, for example cycles of length at least~$\ell$, cycles of length~$p$ modulo~$q$, cycles intersecting a prescribed set of vertices at least~$t$ times, and cycles contained in given $\mathbb{Z}_2$-homology classes in a graph embedded on a fixed surface. 
    Our main result allows us to 
    prove a duality theorem for cycles satisfying a fixed set of finitely many such properties.
\end{abstract}



\maketitle

\section{Introduction}

A classical theorem of Erd\H{o}s and P\'{o}sa~\cite{ErdosP1965} states that every graph contains either~$k$ vertex-disjoint cycles 
or a vertex set of size at most~${\mathcal{O}(k \log k)}$ that hits all cycles of~$G$. 
Such a theorem does not hold if we restrict to odd cycles; Lov\'{a}sz and Schrijver~(see~\cite{Thomassen1988}) 
found a class of graphs having no two vertex-disjoint odd cycles and no small vertex set hitting all odd cycles. 

In the setting of odd cycles, Reed~\cite{Reed1999} obtained an analogue of the theorem of Erd\H{o}s and P\'{o}sa by relaxing the ``vertex-disjoint'' condition. 
A \emph{half-integral packing} is a set of subgraphs such that no vertex is contained in more than two of them. 
Reed~\cite{Reed1999} proved that there is a function~$f$ such that 
every graph has a half-integral packing of at least~$k$ odd cycles 
or has a vertex set of size at most~${f(k)}$ hitting all odd cycles. 
As an easy corollary of Reed's result, given a graph whose edges are labelled with~$\mathbb{Z}_2$, there is a half-integral packing of at least~$k$ cycles, each of non-zero total weight, 
or a vertex set of size at most~${f(k)}$ hitting all such cycles. 

Very recently, Thomas and Yoo~\cite{YooR2020} extended this result to arbitrary abelian groups: 
they showed that there is a function~$f$ such that given a graph whose edges are labelled by an abelian group, there is a half-integral packing of at least~$k$ cycles each of non-zero total weight, 
or a vertex set of size at most~${f(k)}$ hitting all such cycles. 

Kakimura and Kawarabayashi~\cite{KakimuraK2013} proved a different kind of strengthening of the theorem of Reed. 
They showed that there is a function~$f$ such that every graph 
contains a half-integral packing of~$k$ odd cycles each of which intersects a prescribed set~$S$ of vertices 
or a vertex set of size at most~${f(k)}$ hitting all such cycles. 
When~$S$ is the entire vertex set of the graph, this is equivalent to the theorem of Reed. 
This result can be encoded in the setting of graphs labelled with two abelian groups~$\mathbb{Z}$ and~$\mathbb{Z}_2$: 
in~$\mathbb{Z}$ we label each edge incident with a vertex in~$S$ with~$1$, and all other edges with~$0$, 
and in~$\mathbb{Z}_2$ we label each edge with~$1$. 
The cycles which are of non-zero total weight with respect to both of these group labellings are precisely the cycles described by Kakimura and Kawarabayashi. 
Note that, since two groups are required for the encoding, this result is not covered by the previously mentioned result of Thomas and Yoo. 

Our main theorem generalises all of these results to the setting of cycles in graphs labelled with a bounded number of abelian groups, whose values avoid a bounded number of elements of each group.
For an abelian group~${\Gamma}$ and a graph~$G$, a function~${\gamma \colon E(G) \to \Gamma}$ is called a \emph{$\Gamma$-labelling} of~$G$. 
The \emph{$\gamma$-value} of a subgraph~$H$ of~$G$ is the sum of~${\gamma(e)}$ over all edges~$e$ in~$H$. 
For an integer~$m$, we write~${[m]}$ for the set of positive integers at most~$m$.
We call a set of vertices which hits all subgraphs in a set~$\mathcal{H}$ a \emph{hitting set for~$\mathcal{H}$}.

\begin{restatable}{theorem}{mainthm}
    \label{thm:main}
    For every pair of positive integers~$m$ and~$\omega$, there is a function~${f_{m,\omega} \colon \mathbb{N} \to \mathbb{N}}$ satisfying the following property. 
    For each~${i \in [m]}$, let~$\Gamma_i$ be an abelian group, and let~$\Omega_i$ be a subset of~$\Gamma_i$. 
    Let~$G$ be a graph and for each~${i \in [m]}$, let~${\gamma_i}$ be a~${\Gamma_i}$-labelling of~$G$, and let~${\mathcal{O}}$ be the set of all cycles of~$G$ whose $\gamma_i$-value is in~${\Gamma_i \setminus \Omega_i}$ for all~${i \in [m]}$.
    If~${\abs{\Omega_i} \leq \omega}$ for all~${i \in [m]}$, 
    then for all~${k \in \mathbb{N}}$ there exists 
    either a half-integral packing of~$k$ cycles in~$\mathcal{O}$, 
    or a hitting set for~$\mathcal{O}$ of size at most~${f_{m,\omega}(k)}$. 
\end{restatable}

We point out that the function~${f_{m, \omega}}$ does not depend on the choice of groups~${\Gamma_i}$ or the subsets~$\Omega_i$. 
For a graph labelled with a single abelian group~$\Gamma$, 
Wollan~\cite{Wollan2011} showed that if~$\Gamma$ has no element of order~$2$, then there are arbitrarily many vertex-disjoint cycles of non-zero $\gamma$-value, or a hitting set of bounded size for the $\gamma$-non-zero cycles.
However, if~$\Gamma$ has an element of order~$2$, then as in the case of odd cycles, this does not hold~\cite{Thomassen1988, Reed1999}. 
Hence, even when~${m = \omega = 1}$, the half-integral condition cannot be removed in Theorem~\ref{thm:main}. 

If the number~$m$ of given groups~${\Gamma_i}$ is unbounded or the size of~${\Omega_i}$ is unbounded in Theorem~\ref{thm:main}, then such a function~$f_{m, \omega}$ does not exist, and moreover, for every integer $n\ge 2$, a ${(1/n)}$-integral analogue of the Erd\H{o}s-P\'{o}sa theorem does not hold. 
We discuss this in Section~\ref{sec:overview}. 

\medskip

We now present some corollaries which help illustrate the power of this theorem. In the following corollaries, we use the function $f_{m,\omega}$ in Theorem~\ref{thm:main}.
As already mentioned, the cycles in a graph~$G$ which intersect a prescribed set of vertices can be encoded as precisely the non-zero cycles with respect to the $\mathbb{Z}$-labelling which assigns value~$1$ to edges incident with vertices in~$S$ and~$0$ to all other edges.
If instead each edge of~$G$ is assigned the integer that is the number of its endvertices which lie in~$S$, then the cycles of total weight at least~${2t}$ are precisely the cycles which intersect the set~$S$ at least~$t$ times. 
Using multiple labellings with~$\mathbb{Z}$, we can encode the set of cycles which intersect each of a bounded number of sets at least~$t$ times each. 
Thus we obtain our first corollary.

\begin{corollary} 
   \label{cor:S_i-intersecting}
    Let~$m$ and~$t$ be positive integers.
    For each~${i \in [m]}$, let~${S_i}$ be a subset of the vertices of a graph~$G$ and~${t_i \in [t]}$, and let~$\mathcal{O}$ be the set of all cycles of~$G$ containing at least~$t_i$ vertices of~$S_i$ for all~${i \in [m]}$. 
    Then for all~${k \in \mathbb{N}}$ there exists 
    either a half-integral packing of~$k$ cycles in~$\mathcal{O}$, 
    or a hitting set for~$\mathcal{O}$ of size at most~${f_{m,t}(k)}$. 
\end{corollary}

In fact, the construction above can be generalised, allowing us to convert a group labelling of the vertices of a graph to a group labelling of its edges. 
For an abelian group~$\Gamma$ and a graph~$G$, a \emph{$\Gamma$-vertex-labelling of~$G$} is a function~${\gamma \colon V(G) \to \Gamma}$, and the \emph{$\gamma$-value} of a subgraph~$H$ of~$G$ is the sum of~${\gamma(v)}$ over all vertices in~$H$. 
In Section~\ref{sec:overview}, we will discuss in detail how such a conversion works in general, and thus obtain the following corollary. 

\begin{restatable}{corollary}{vxlabel}
    \label{cor:vxlabelling}
    Let~$m$ and~$\omega$ be positive integers. 
    For each~${i \in [m]}$, let~$\Gamma_i$ be an abelian group and let~$\Omega_i$ be a subset of~$\Gamma_i$. 
    Let~$G$ be a graph, and for each~${i \in [m]}$, let~${\gamma_i}$ be a~${\Gamma_i}$-vertex-labelling of~$G$, and let~${\mathcal{O}}$ be the set of all cycles of~$G$ whose $\gamma_i$-value is in~${\Gamma_i \setminus \Omega_i}$ for all~${i \in [m]}$.
    If~${\abs{\Omega_i} \leq \omega}$ for all~${i \in [m]}$, 
    then for all~${k \in \mathbb{N}}$ there exists 
    either a half-integral packing of~$k$ cycles in~$\mathcal{O}$, 
    or a hitting set for~$\mathcal{O}$ of size at most~${f_{m,\omega}(k)}$.
\end{restatable}

Our next corollary relates to graphs labelled with a fixed finite abelian group, where we obtain a similar result for the set of cycles of any specified value. 
In particular, this shows that for every pair of positive integers~${p}$ and~${q}$, cycles of length~${p}$ modulo~${q}$ satisfy a half-integral analogue of the Erd\H{o}s-P\'{o}sa theorem. 
Dejter and Neumann-Lara~\cite{DejterN1988} showed that without the half-integral relaxation, the analogous Erd\H{o}s-P\'{o}sa type result fails for cycles of length~$p$ modulo~$q$ whenever the least common multiple~${\lcm(p,q)}$ of~$p$ and~$q$ is divisible by~$2p$ (see also~\cite{Wollan2011}). 
When this condition is not met, Gollin, Hendrey, Kwon, Oum, and Yoo~\cite{GollinHKOY21} show that the half-integral relaxation is required for an Erd\H{o}s-P\'{o}sa type result in this setting if and only if~${\lcm(p,q)/p}$ is divisible by three distinct primes. 

\begin{corollary}
    \label{cor:specific-value-hiEP}
    Let~${\Gamma}$ be a finite abelian group and let~${g \in \Gamma}$.
    Let~$G$ be a graph, and~$\gamma$ be a~${\Gamma}$-labelling of~$G$, 
    and let~$\mathcal{O}$ be the set of all cycles of~$\gamma$-value~$g$. 
    Then for all~${k \in \mathbb{N}}$ there exists 
    either a half-integral packing of~$k$ cycles in~$\mathcal{O}$, 
    or a hitting set for~$\mathcal{O}$ of size at most~${f_{1, \abs{\Gamma}-1}(k)}$. 
\end{corollary}

Our next corollary relates to graphs embedded on a fixed compact surface, where we obtain a similar result for the set of cycles contained in any given set of (first) $\mathbb{Z}_2$-homology classes. 
Huynh, Joos, and Wollan~\cite{HuynhJW2017} proved that for graphs embedded on a fixed surface, cycles not homologous to zero in the $\mathbb{Z}$-homology group satisfy a half-integral analogue of the Erd\H{o}s-P\'{o}sa theorem. 
They used a different type of graph labelling, called a directed $\Gamma$-labelling.
With (undirected) $\Gamma$-labellings, we can do the same thing with respect to $\mathbb{Z}_2$-homology classes. Since a compact surface has a finite abelian group as its $\mathbb{Z}_2$-homology group, we can obtain a half-integral analogue of the Erd\H{o}s-P\'{o}sa theorem for cycles contained in any given set of  $\mathbb{Z}_2$-homology classes. 
We discuss this further in Section~\ref{sec:overview}. 

\begin{corollary}
    \label{cor:surface}
    Let~$\Sigma$ be a compact surface with $\mathbb{Z}_2$-homology group $\Gamma$ and let~${\mathcal{C}}$ be a set of $\mathbb{Z}_2$-homology classes of~$\Sigma$. 
    Let~$G$ be a graph embedded on~$\Sigma$, and let~$\mathcal{O}$ be the set of all cycles of~$G$ 
    whose $\mathbb{Z}_2$-homology classes are contained in~$\mathcal{C}$.
    Then for all~${k \in \mathbb{N}}$ there exists 
    either a half-integral packing of~$k$ cycles in~$\mathcal{O}$, 
    or a hitting set for~$\mathcal{O}$ of size at most~${f_{1, \abs{\Gamma}-\abs{\mathcal{C}}}(k)}$. 
\end{corollary}

One nice feature of our main theorem is that it allows us to combine various properties of cycles together to obtain new results, 
as long as we take a bounded number of properties and can encode each of them with a bounded number of group labellings. 
Thus, we could combine any subset of these corollaries together and obtain a result of the same form. 

\medskip

Huynh, Joos, and Wollan~\cite{HuynhJW2017} obtained a result similar to our main theorem for graphs with two directed group labellings, where the value of an edge is inverted if it is traversed in the reverse direction.
They showed that a half-integral analogue of the Erd\H{o}s-P\'{o}sa theorem holds for cycles whose values are non-zero in each coordinate. 
They conjectured that their result can be extended to graphs with more than two directed labellings. 
Because the~${\Gamma}$-labellings in this paper are equivalent to directed ${\Gamma}$-labellings when all elements of~${\Gamma}$ have order~$2$, 
Theorem~\ref{thm:main} implies that the conjecture of Huynh, Joos, and Wollan hold for graphs with a fixed number of directed labellings with such groups. 
The conjecture otherwise remains open, although there is a large overlap between the motivation of the conjecture and the consequences of our main theorem. 
We discuss directed group labellings in more detail in Section~\ref{sec:conclusion}.

\medskip

The structure of this paper is as follows. 
In Section~\ref{sec:prelim}, we introduce preliminary concepts, and we give a high-level overview of the proof of our main theorem and present proofs of corollaries in Section~\ref{sec:overview}.
In Section~\ref{sec:pack}, we define a packing function and provide its application.
In Section~\ref{sec:cleanwalls}, we define a concept of a clean wall which is well-behaved for each of the labellings~$\gamma_i$. 
In Section~\ref{sec:handles}, we prove a key lemma to find many vertex-disjoint paths attached to a wall.
In Section~\ref{sec:abelian}, we prove useful lemmas on a product of abelian groups that will be used in the last step of the main theorem. 
Section~\ref{sec:handles2cycles} discusses how to obtain a desired cycle from a wall together with attached disjoint paths. 
We prove our main result in Section~\ref{sec:main}, 
and we discuss some open problems in Section~\ref{sec:conclusion}.

\section{Preliminaries}
\label{sec:prelim}

In this paper, all graphs are undirected simple graphs having no loops and multiple edges. For every abelian group, we regard its operation as an additive operation and denote its zero by~$0$.
Even though we work on simple graphs, all the results are extended to multigraphs; given a multigraph instance, we can take a subdivision to produce an equivalent simple graph with group value~$0$ on new edges. 

For an integer~$m$, we write~${[m]}$ for the set of positive integers at most~$m$. 

Let~$G$ be a graph. 
We denote by~${V(G)}$ and~${E(G)}$ the vertex set and the edge set of~$G$, respectively. 
For a vertex set~$A$ of~$G$, we denote by~${G - A}$ the graph obtained from~$G$ by deleting all the vertices in~$A$ and all edges incident with vertices in~$A$, 
and denote by~${G[A]}$ the subgraph of~$G$ induced by~$A$, which is~${G-(V(G)\setminus A)}$.
If~${A = \{v\}}$, then we write~${G - v}$ for~${G - A}$. 
For an edge~$e$ of~$G$, we denote by~${G - e}$ the graph obtained by deleting~$e$. 
For two graphs~$G$ and~$H$, let 
\[
    {G \cup H := (V(G) \cup V(H), E(G) \cup E(H))}
    \ \textnormal{ and } \ 
    {G \cap H := (V(G) \cap V(H), E(G) \cap E(H))}. 
\]
For a set~$\mathcal{G}$ of graphs, we denote by~${\bigcup \mathcal{G}}$ the union of the graphs in~$\mathcal{G}$. 

For an integer~${t}$, a graph~$G$ is \emph{$t$-connected} if it has more than~$t$ vertices and~${G-S}$ is connected for all vertex sets~$S$ with~${\abs{S}<t}$. 

\emph{Subdividing} an edge~$uv$ in a graph~$G$ is an operation that yields a graph containing one new vertex~$w$, and with an edge set replacing~$uv$ by two new edges,~$uw$ and~$wv$. 
A graph~$H$ is a \emph{subdivision} of a graph~$G$ if~$H$ can be obtained from~$G$ by subdividing edges repeatedly. 

Let~$A$ and~$B$ be vertex sets of~$G$. 
An \emph{${(A, B)}$-path} is a path from a vertex in~$A$ to a vertex in~$B$ such that all internal vertices are not contained in~${A \cup B}$. 
We also denote an~${(A,A)}$-path as an \emph{$A$-path}. 
For a subgraph~$H$ of~$G$, we shortly write as an $H$-path for a ${V(H)}$-path.
A path is \emph{$A$-intersecting} if it contains a vertex of~$A$. 

For a graph~$G$, we denote by~$\branch(G)$ the set of all vertices of~$G$ whose degrees are not equal to~$2$.
A \emph{corridor} of a graph~$G$ is a $\branch(G)$-path of length at least~$1$.

\begin{remark}
    Every graph in which no block is a cycle is the edge-disjoint union of its corridors. 
\end{remark}

For a family~${\mathcal{F} = ( x_i \colon i \in I )}$ we write $\abs{\mathcal{F}}=\abs{I}$, called the \emph{size} of~$\mathcal{F}$.

\subsection{Walls}

Let~${c,r \geq 3}$ be integers. 
The \emph{elementary $(c,r)$-wall $W_{c,r}$} is the graph obtained from the graph on the vertex set~${[2c] \times [r]}$ whose edge set is 
\begin{align*}
    \left\{ (i,j) (i+1,j) \, \colon \, i \in [2c-1],\, j \in [r] \right\} 
    \cup 
    \left\{ (i,j) (i,j+1) \, \colon \, i \in [2c],\, j \in [r-1],\, i+j \textnormal{ is odd} \right\}
\end{align*}
by deleting both degree~$1$ vertices. 

\begin{figure}[h]
    \centering
    \begin{tikzpicture}[scale=0.5, decoration={markings, mark=at position 0.5 with {\arrow{>}}}]
        \tikzstyle{w}=[circle,draw,fill=black!50,inner sep=0pt,minimum width=3pt]

        \draw[gray] (1,0)--(9,0);
        \draw[gray] (0,6)--(8,6);
        
        \foreach \y in {1,2,3,4,5}{
            \draw[gray] (0,\y)--(9,\y);
        }
        
        \foreach \x in {0, 2, 4, 6, 8}{
            \foreach \y in {1,3,5}{
                \draw[gray] (\x, \y+1)-- (\x, \y);
                \draw[gray] (\x+1, \y)-- (\x+1, \y-1);
            }
        }
		
        \foreach \x in {2, 4, 6}{
            \foreach \y in {3,5}{
                \draw[very thick] (\x, \y+1)-- (\x, \y);
                \draw[very thick] (\x+1, \y)-- (\x+1, \y-1);
                \draw[very thick] (\x, \y)-- (\x+1, \y);
                \draw[very thick] (\x+1, \y-1)-- (\x, \y-1);
            }
        }
	    
        \foreach \x in {2,4,6}{
            \foreach \y in {1}{
                \draw[very thick] (\x, \y+1)-- (\x, \y);
                \draw[very thick] (\x+1, \y)-- (\x+1, \y-1);
                \draw[very thick] (\x, \y)-- (\x+1, \y);	
                \draw (\x+1, \y-1)-- (\x, \y-1);	
            }
        }
        \foreach \x in {2,4}{
            \draw[very thick] (\x+1, 7-1)-- (\x, 7-1);	
        }
        \foreach \y in {0,1,2,3,4,5}{
        \draw[very thick] (3, \y)--(5,\y);
        \draw[very thick] (5, \y)--(7,\y);
        }
        \draw[very thick] (3, 6)--(6,6);
        \node at (2, .5) {};
        \node at (4, .5) {};
        \node at (1, 1.5) {};
        \node at (3, 1.5) {};
        \foreach \x in {0,1,...,9}{
            \foreach \y in {1,2,...,5} {
                \node at (\x,\y) [w] {};
            }
        }
        \foreach \x in {0,1,...,8}{
            \node at (\x,6) [w] {};
        }
        \foreach \x in {1,2,...,9}{
            \node at (\x,0) [w] {};
        }
    \end{tikzpicture}
    \qquad
    \begin{tikzpicture}[scale=0.5, decoration={markings, mark=at position 0.5 with {\arrow{>}}}]
        \tikzstyle{w}=[circle,draw,fill=black!50,inner sep=0pt,minimum width=3pt]
        \foreach \x in {0,...,9}{
            \ifthenelse{\x>0}{
                \pgfmathtruncatemacro{\start}{0};
            }{
                \pgfmathtruncatemacro{\start}{1};
            }
            \ifthenelse{\x<9}{
                \pgfmathtruncatemacro{\t}{6};
            }{
                \pgfmathtruncatemacro{\t}{5};
            }
            \foreach \y in {\start,...,\t} {
                \node at (\x,\y) [w] (v\x\y) {};
            }
        }        
        \foreach \x in {0,1,...,9}{
            \ifthenelse{\x>0}{
                \pgfmathtruncatemacro{\start}{0};
            }{
                \pgfmathtruncatemacro{\start}{1};
            }
            \ifthenelse{\x<9}{
                \pgfmathtruncatemacro{\t}{6};
            }{
                \pgfmathtruncatemacro{\t}{5};
            }          
            \foreach \y in {\start,...,\t} {
                \pgfmathtruncatemacro{\nextx}{\x+1}
                \ifthenelse{\x<8}{ 
                    \draw [gray](v\x\y)--(v\nextx\y);
                }{}
                \ifthenelse{\x=8 \and \y<6}{ 
                    \draw [gray](v\x\y)--(v\nextx\y);
                }{}                
                \pgfmathtruncatemacro{\nexty}{\y+1}
                \pgfmathtruncatemacro{\sum}{\x+\y}
                \ifthenelse{\isodd{\sum}\and \y<6}{
                    \draw [gray](v\x\y)--(v\x\nexty);
                }{}
            }
        }
        \foreach \x in {1,...,8}{
            \ifthenelse{\x>1}{
                \pgfmathtruncatemacro{\start}{1};
            }{
                \pgfmathtruncatemacro{\start}{2};
            }
            \ifthenelse{\x<8}{
                \pgfmathtruncatemacro{\t}{5};
            }{
                \pgfmathtruncatemacro{\t}{4};
            }          
            \foreach \y in {\start,...,\t} {
                \pgfmathtruncatemacro{\nextx}{\x+1}
                \ifthenelse{\x<7}{ 
                    \draw [very thick](v\x\y)--(v\nextx\y);
                }{}
                \ifthenelse{\x=7 \and \y<5}{ 
                    \draw [very thick](v\x\y)--(v\nextx\y);
                }{}                
                \pgfmathtruncatemacro{\nexty}{\y+1}
                \pgfmathtruncatemacro{\sum}{\x+\y}
                \ifthenelse{\isodd{\sum}\and \y<5}{
                    \draw [very thick](v\x\y)--(v\x\nexty);
                }{}
            }
        }
    \end{tikzpicture}
    \caption{An elementary ${(5, 7)}$-wall~$W$ depicted on the left with a $3$-column-slice, and depicted on the right with a ${(4,5)}$-subwall that is ${\branch(W)}$-anchored.}
    \label{fig:wall}
\end{figure}
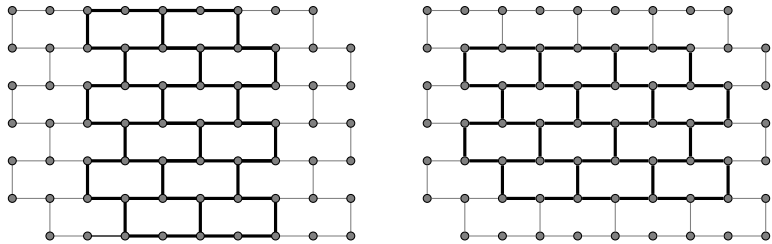

For ${j \in [r]}$, the \emph{$j$-th row}~$R_j$ of~$W_{c,r}$ is the path~${W_{c,r} \big[ \big\{ (i, j)\in V(W_{c,r}) \, \colon \, i \in [2c] \big\} \big]}$. 
For ${i \in [c]}$, the \emph{$i$-th column}~$C_i$ of~$W_{c,r}$ is the path~${W_{c,r} \big[ \big\{ (i',j)\in V(W_{c,r}) \, \colon \, i' \in \{ 2i - 1, 2i \}, \, j \in [r] \big\} \big]}$. 

A \emph{$(c,r)$-wall} is a subdivision~$W$ of the elementary $(c,r)$-wall. 
If~$W$ is a~$(c,r)$-wall for some suitable integers~$c$ and~$r$, then we say~$W$ is a \emph{wall of order~$\min\{c,r\}$}. 
We call a branch vertex corresponding to the vertex~${(i,j)}$ of the elementary wall a \emph{nail} of~$W$, and denote by~$N^W$ the set of nails of~$W$. 

\begin{remark}
    \label{rmk:wall3connected}
    Any wall is a subdivision of a $3$-connected planar graph. 
\end{remark}

For a subgraph~$H$ of the elementary wall, we denote by~$H^W$ the subgraph of~$W$ corresponding to a subdivision of~$H$.  
We call~$R_j^W$ or $C_i^W$ the \emph{$j$-th row} or \emph{$i$-th column} of~$W$, respectively. 
A subgraph~$W'$ of a wall~$W$ that is itself a wall is called a \emph{subwall of~$W$}. 
For a set~$S$ of vertices, we say a wall~$W$ is \emph{$S$-anchored} if~${N^W \subseteq S}$. 
We observe the following.

\begin{remark}
    \label{rmk:nicesubwall}
    Let~$W$ be a~$(c,r)$-wall for integers~${c, r \geq 5}$. 
    Then~$W$ contains a~$(c-1,r-2)$-subwall~$W'$ which is~$\branch(W)$-anchored. 
    See Figure~\ref{fig:wall}. 
    \qed
\end{remark}

For an integer~${c \geq 3}$, we call a subwall~$W'$ of a wall~$W$ a \emph{$c$-column-slice of~$W$} if the set of nails of~$W'$ is exactly~${N^W \cap V(W')}$,
there is a column of~$W'$ which is a column of~$W$, 
and~$W'$ has exactly~$c$ columns, see Figure~\ref{fig:wall} for an example. 
Similarly, for an integer~${r \geq 3}$, we call a subwall~$W'$ of a wall~$W$ an \emph{$r$-row-slice of~$W$} if the set of nails of~$W'$ is exactly~${N^W \cap V(W')}$,
there is a row of~$W'$ which is a row of~$W$, 
and~$W'$ has exactly~$r$ rows. 
Note that in an $r$-row-slice~$W'$ of~$W$, depending on the location, the first column of~$W'$ may be in the last column of~$W$ by the definition of a wall. 
See Figure~\ref{fig:rowslice} for an illustration.

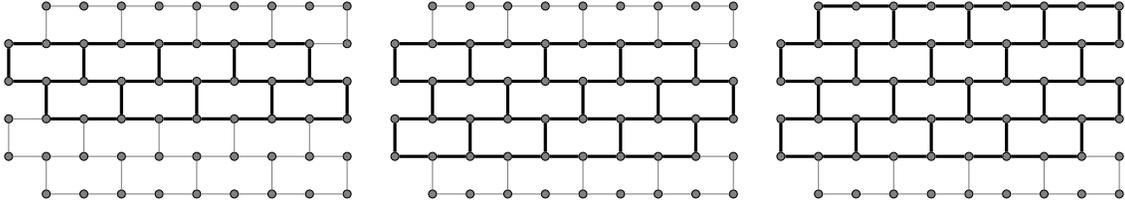
\begin{figure}[h]
    \centering
    \begin{tikzpicture}[scale=0.5, decoration={markings, mark=at position 0.5 with {\arrow{>}}}]
        \tikzstyle{w}=[circle,draw,fill=black!50,inner sep=0pt,minimum width=3pt]
        \foreach \x in {0,...,9}{
            \ifthenelse{\x>0}{
                \pgfmathtruncatemacro{\start}{0};
                \pgfmathtruncatemacro{\t}{5};
            }{
                \pgfmathtruncatemacro{\start}{1};
                \pgfmathtruncatemacro{\t}{4};
            }
            \foreach \y in {\start,...,\t} {
                \node at (\x,\y) [w] (v\x\y) {};
            }
        }        
        \foreach \x in {0,1,...,9}{
            \ifthenelse{\x>0}{
                \pgfmathtruncatemacro{\start}{0};
                \pgfmathtruncatemacro{\t}{5};
            }{
                \pgfmathtruncatemacro{\start}{1};
                \pgfmathtruncatemacro{\t}{4};
            }        
            \foreach \y in {\start,...,\t} {
                \pgfmathtruncatemacro{\nextx}{\x+1}
                \ifthenelse{\x<9}{ 
                    \draw [gray](v\x\y)--(v\nextx\y);
                }{}
                \pgfmathtruncatemacro{\nexty}{\y+1}
                \pgfmathtruncatemacro{\sum}{\x+\y}
                \ifthenelse{\isodd{\sum}\and \y<5}{
                    \draw [gray](v\x\y)--(v\x\nexty);
                }{}
            }
        }
        \foreach \x in {0,...,9}{
            \ifthenelse{\x>0}{
                \pgfmathtruncatemacro{\start}{2};
            }{
                \pgfmathtruncatemacro{\start}{3};
            }
            \ifthenelse{\x<9}{
                \pgfmathtruncatemacro{\t}{4};
            }{
                \pgfmathtruncatemacro{\t}{3};
            }
            \foreach \y in {\start,...,\t} {
                \pgfmathtruncatemacro{\nextx}{\x+1}
                \ifthenelse{\x<8 \and \x>0}{ 
                    \draw [very thick](v\x\y)--(v\nextx\y);
                }{}
                \ifthenelse{\x=8 \and \y<4}{ 
                    \draw [very thick](v\x\y)--(v\nextx\y);
                }{}   
                \ifthenelse{\x=0 \and \y>2}{ 
                    \draw [very thick](v\x\y)--(v\nextx\y);
                }{}             
                \pgfmathtruncatemacro{\nexty}{\y+1}
                \pgfmathtruncatemacro{\sum}{\x+\y}
                \ifthenelse{\isodd{\sum}\and \y<4}{
                    \draw [very thick](v\x\y)--(v\x\nexty);
                }{}
            }
        }
    \end{tikzpicture}
    \quad
    \begin{tikzpicture}[scale=0.5, decoration={markings, mark=at position 0.5 with {\arrow{>}}}]
        \tikzstyle{w}=[circle,draw,fill=black!50,inner sep=0pt,minimum width=3pt]
        \foreach \x in {0,...,9}{
            \ifthenelse{\x>0}{
                \pgfmathtruncatemacro{\start}{0};
                \pgfmathtruncatemacro{\t}{5};
            }{
                \pgfmathtruncatemacro{\start}{1};
                \pgfmathtruncatemacro{\t}{4};
            }
            \foreach \y in {\start,...,\t} {
                \node at (\x,\y) [w] (v\x\y) {};
            }
        }        
        \foreach \x in {0,1,...,9}{
            \ifthenelse{\x>0}{
                \pgfmathtruncatemacro{\start}{0};
                \pgfmathtruncatemacro{\t}{5};
            }{
                \pgfmathtruncatemacro{\start}{1};
                \pgfmathtruncatemacro{\t}{4};
            }        
            \foreach \y in {\start,...,\t} {
                \pgfmathtruncatemacro{\nextx}{\x+1}
                \ifthenelse{\x<9}{ 
                    \draw [gray](v\x\y)--(v\nextx\y);
                }{}
                \pgfmathtruncatemacro{\nexty}{\y+1}
                \pgfmathtruncatemacro{\sum}{\x+\y}
                \ifthenelse{\isodd{\sum}\and \y<5}{
                    \draw [gray](v\x\y)--(v\x\nexty);
                }{}
            }
        }
        \foreach \x in {0,...,9}{
            \ifthenelse{\x<9}{
                \pgfmathtruncatemacro{\start}{1};
                \pgfmathtruncatemacro{\t}{4};
            }{
                \pgfmathtruncatemacro{\start}{2};
                \pgfmathtruncatemacro{\t}{3};
            }
            \foreach \y in {\start,...,\t} {
                \pgfmathtruncatemacro{\nextx}{\x+1}
                \ifthenelse{\x<8}{ 
                    \draw [very thick](v\x\y)--(v\nextx\y);
                }{}
                \ifthenelse{\x=8 \and \y<4 \and \y>1}{ 
                    \draw [very thick](v\x\y)--(v\nextx\y);
                }{}                
                \pgfmathtruncatemacro{\nexty}{\y+1}
                \pgfmathtruncatemacro{\sum}{\x+\y}
                \ifthenelse{\isodd{\sum}\and \y<4}{
                    \draw [very thick](v\x\y)--(v\x\nexty);
                }{}
            }
        }
    \end{tikzpicture}
    \quad
    \begin{tikzpicture}[scale=0.5, decoration={markings, mark=at position 0.5 with {\arrow{>}}}]
        \tikzstyle{w}=[circle,draw,fill=black!50,inner sep=0pt,minimum width=3pt]
        \foreach \x in {0,...,9}{
            \ifthenelse{\x>0}{
                \pgfmathtruncatemacro{\start}{0};
                \pgfmathtruncatemacro{\t}{5};
            }{
                \pgfmathtruncatemacro{\start}{1};
                \pgfmathtruncatemacro{\t}{4};
            }
            \foreach \y in {\start,...,\t} {
                \node at (\x,\y) [w] (v\x\y) {};
            }
        }        
        \foreach \x in {0,1,...,9}{
            \ifthenelse{\x>0}{
                \pgfmathtruncatemacro{\start}{0};
                \pgfmathtruncatemacro{\t}{5};
            }{
                \pgfmathtruncatemacro{\start}{1};
                \pgfmathtruncatemacro{\t}{4};
            }        
            \foreach \y in {\start,...,\t} {
                \pgfmathtruncatemacro{\nextx}{\x+1}
                \ifthenelse{\x<9}{ 
                    \draw [gray](v\x\y)--(v\nextx\y);
                }{}
                \pgfmathtruncatemacro{\nexty}{\y+1}
                \pgfmathtruncatemacro{\sum}{\x+\y}
                \ifthenelse{\isodd{\sum}\and \y<5}{
                    \draw [gray](v\x\y)--(v\x\nexty);
                }{}
            }
        }
        \foreach \x in {0,...,9}{
            \ifthenelse{\x<9}{
                \pgfmathtruncatemacro{\start}{1};
            }{
                \pgfmathtruncatemacro{\start}{2};
            }
            \ifthenelse{\x>0}{
                \pgfmathtruncatemacro{\t}{5};
            }{
                \pgfmathtruncatemacro{\t}{4};
            }
            \foreach \y in {\start,...,\t} {
                \pgfmathtruncatemacro{\nextx}{\x+1}
                \ifthenelse{\x<8}{ 
                    \draw [very thick](v\x\y)--(v\nextx\y);
                }{}
                \ifthenelse{\x=8 \and  \y>1}{ 
                    \draw [very thick](v\x\y)--(v\nextx\y);
                }{}                
                \pgfmathtruncatemacro{\nexty}{\y+1}
                \pgfmathtruncatemacro{\sum}{\x+\y}
                \ifthenelse{\isodd{\sum}\and \y<5}{
                    \draw [very thick](v\x\y)--(v\x\nexty);
                }{}
            }
        }
    \end{tikzpicture}
    \caption{A $3$-row-slice $W_3$, a $4$-row-slice $W_4$, and a $5$-row-slice $W_5$ of a $(5,6)$-wall~$W$. Notice that the first column of $W_3$ is in the first column of $W$ but the first column of $W_4$ or $W_5$ is in the last column of $W$.}
    \label{fig:rowslice}
\end{figure}

Given a wall~$W$, a \emph{$W$-handle} is a non-trivial $W$-path whose endvertices are degree-$2$ nails of~$W$ contained in the union of the first and last column of~$W$.

Let~$W$ be a ${(c,r)}$-wall, let~$W'$ be a $c'$-column-slice of~$W$ for some~${3 \leq c' \leq c}$, 
and let~$P$ be a $W$-handle. 
We define the \emph{row-extension of~$P$ to~$W'$ in~$W$} as the unique non-trivial $W'$-path in~${P \cup \bigcup \{ R_i^W \colon i \in [r]\}}$. 
Note that such a~$P$ is a $W'$-handle. 
For a set~$\mathcal{P}$ of vertex-disjoint $W$-handles, we define the \emph{row-extension of~$\mathcal{P}$ to~$W'$ in~$W$} to be the set of row-extensions of the paths in~$\mathcal{P}$ to~$W'$ in~$W$. 
Note that these $W'$-handles are vertex-disjoint.

\subsection{Linkages, separations, and tangles}

Let~$G$ be a graph.
For vertex sets~$A$ and~$B$ in~$G$, 
a set~$\mathcal{P}$ of vertex-disjoint ${(A, B)}$-paths of~$G$ is called a \emph{linkage} from~$A$ to~$B$, 
and its \emph{order} is defined to be~$\abs{\mathcal{P}}$. 
A \emph{separation} of~$G$ is a pair~${(A, B)}$ of subsets of~$V(G)$ such that~${G[A] \cup G[B] = G}$. 
Its \emph{order} is defined to be~$\abs{A \cap B}$. 
We will use Menger's theorem. 

\begin{theorem}[Menger~\cite{Menger27}]
    \label{thm:menger}  
    Let~$A$ and~$B$ be vertex sets in a graph~$G$, and~$k$ be a positive integer. 
    Then~$G$ contains either a linkage of order~$k$ from~$A$ to~$B$, 
    or a separation~${(A', B')}$ of order less than~$k$ such that~${A \subseteq A'}$ and~${B \subseteq B'}$. 
\end{theorem}

We need a concept of a large wall dominated by a tangle, see~\cite[(2.3)]{RST1994}. 
For a positive integer~$t$, 
a set~$\cT$ of separations of order less than~$t$ is a \emph{tangle of order~$t$} in~$G$ if it satisfies the following. 
\begin{enumerate}
    [label=(\arabic*)]
    \item If ${(A, B)}$ is a separation of~$G$ of order less than~$t$, then~$\cT$ contains exactly one of~${(A, B)}$ and~${(B, A)}$. 
    \item If~${(A_1, B_1), (A_2, B_2), (A_3, B_3) \in \cT}$, then~${G[A_1] \cup G[A_2] \cup G[A_3] \neq G}$. 
\end{enumerate}
Let~$W$ be a wall of order~$g$ with~${g \geq 3}$ in a graph~$G$. 
Let~$\cT_W$ be the set of all separations~${(A, B)}$ of~$G$ of order less than~$g$ such that~${G[B]}$ contains a row of~$W$. 
By the following simple lemma, we may replace the row with the column.
The proof in~\cite{RS1991} is for the grid but one can easily modify it for the wall.

\begin{lemma}[{Robertson and Seymour~\cite[(7.1)]{RS1991}}]
    \label{lem:rowcol}
    Let~$W$ be a wall of order~$g$ in a graph~$G$.
    Let~${(A,B)}$ be a separation of order less than~$g$.
    Then~${G[B]}$ contains a row of~$W$
    if and only if it contains a column of~$W$.
\end{lemma}

Kleitman and Saks (see~\cite[(7.3)]{RS1991}) showed that~$\cT_W$ is a tangle of order~$g$. 
A tangle~$\cT$ in~$G$ \emph{dominates} the wall~$W$ if~${\cT_W \subseteq \cT}$.
    
\begin{theorem}[Robertson, Seymour, and Thomas~\cite{RST1994}]
    \label{thm:wall}
    There exists a function~${f_{\ref*{thm:wall}} \colon \mathbb{N} \to \mathbb{N}}$ 
    such that 
    if~${g \geq 3}$ is an integer and~$\cT$ is a tangle in a graph~$G$ of order at least~${f_{\ref*{thm:wall}}(g)}$,
    then~$\cT$ dominates a ${(g, g)}$-wall~${W}$ in~${G}$. 
\end{theorem}

We will show that if a tangle dominates a wall~$W$, then it also dominates every large subwall of~$W$. 
We first prove the following lemma. 

\begin{lemma}
    \label{lem:wallseparations}
    Let~$W$ be a wall in a graph~$G$ and let~${S}$ be a subset of~${V(G)}$ of size exactly~$t$. 
    For each column~$C^W_x$ and row~$R^W_y$ of~$W$, 
    there are no more than~$t^2$ nails of $W$ which belong to components of~${W-S}$ that do not intersect~${C^W_x \cup R^W_y}$.
\end{lemma}

\begin{proof}
    We proceed by induction on~$t$. 
    The statement is trivial if~${t = 0}$. 
    We may assume that~${G = W}$.
    Let~${S =: \{ s_i \colon i \in [t] \}}$ and 
    let~${T := V(C^W_x\cup R^W_y)}$.
    Suppose there is a vertex~$v$ in~${S \setminus N^W}$, and let~$P$ be the $N^W$-path in~$W$ containing~$v$. 
    If both or neither of the endvertices of~$P$ are in components of~${W-S}$ that intersect~$T$, 
    then we may apply the inductive hypothesis to~${S \setminus \{v\}}$. 
    Otherwise, replacing~$v$ in~$S$ with the unique endvertex of~$P$ which is in a component of~${W-S}$ that intersects~$T$ does not decrease the number of nails in components of~${W-S}$ that do not intersect~$T$. 
    Hence, we may assume that~${S \subseteq N^W}$.

    For each~${i \in [t]}$, let~${c(i)}$ be the integer such that~$s_i$ is in~$C^W_{c(i)}$ and let~${r(i)}$ be the integer such that~$s_i$ is in~$R^W_{r(i)}$. 
    For each~${i \in [t]}$, let~$S^c_i$ be the set of nails~$v$ of~$W$ in~${C^W_{c(i)}-S}$ such that the ${(v,R^W_y)}$-subpath of~$C^W_{c(i)}$ contains the vertex~$s_i$, 
    and let~$S^r_i$ be the set of nails~$v$ of~$W$ in~${R^W_{r(i)}-S}$ such that the ${(v,C^W_x)}$-subpath of~$R^W_{r(i)}$ contains the vertex~$s_i$. 
    Note that for every nail~$v$ of~$W$, if $v$ is in a component of~${W-S}$ not intersecting~$T$, then there exist~${i,j \in [t]}$ such that~${v \in S^c_i \cap S^r_j}$. 
    Also note that~${\abs{S^c_i \cap S^r_j} \leq 2}$, and that if~${\abs{S^c_i \cap S^r_j} > 0}$ and~${i \neq j}$, then~${\abs{S^c_j \cap S^r_i} = 0}$. 
    Furthermore, for~${i \in [t]}$, the nails in~${S^c_i \cap S^r_i}$ are in~${(C^W_{c(i)} \cap R^W_{r(i)}) - s_i}$, so~${\abs{S^c_i \cap S^r_i} \leq 1}$. 
    It follows that the number of nails of~$W$ which are in components of~${W-S}$ that do not intersect~${C^W_x \cup R^W_y}$ is at most~${2{\binom{t}{2}} + t = t^2}$. 
\end{proof}

\begin{lemma}
    \label{lem:dominatedsubwall}
    Let~${w \geq t \geq 3}$ be integers, let~$W$ be a wall of order~$w$ in a graph~$G$, and let~$\mathcal{T}$ be a tangle dominating~$W$. 
    If~$W'$ is a subwall of~$W$ of order~$t$ and $\abs{N^{W'}\cap N^W}> (2t-1)(t-1)$, 
    then~$\mathcal{T}$ dominates~$W'$. 
    
    In particular, if $W'$ is $N^W$-anchored, then~$\mathcal{T}$ dominates~$W'$.
\end{lemma}

\begin{proof}
    Suppose for a contradiction that~$\mathcal{T}$ does not dominate~$W'$. 
    Then~$G$ has a separation~${(A,B)}$ of order less than~${t}$ such that~${G[B]}$ contains some row~$R'$ of~$W'$ 
    and~${(A,B) \notin \mathcal{T}}$. 
    The order of~$\mathcal{T}$ is at least~$w$, so~$\mathcal{T}$ contains ${(B,A)}$ and therefore~${G[A]}$ contains some row~$R$ of~$W$.
    Let~${S := A \cap B}$. 
    Since~$W$ has more than~$\abs{S}$ columns and~$R$ intersects each of them, ${G[A]}$ contains some column~$C$ of~$W$, and similarly~${G[B]}$ contains some column~$C'$ of~$W'$. 
    By Lemma~\ref{lem:wallseparations} applied to~$W$, there are at most~${(t-1)^2}$ nails of~$W$ in components of~${G-S}$ which do not intersect~${R \cup C}$. 
    Similarly, there are at most~${(t-1)^2}$ nails of~$W'$ in components of~${G-S}$ which do not intersect~${R'\cup C'}$. 
    Since~${\abs{N^{W'} \cap N^W} > (2t-1)(t-1) \geq 2(t-1)^2 + \abs{S}}$, 
    there is a vertex in~${(N^{W'} \cap N^W) \setminus S}$
    that is in a component of~${G-S}$ intersecting~${R \cup C}$
    and also in a component of~${G-S}$ intersecting~${R' \cup C'}$.
    However, every component of~${G-S}$ intersecting~${R \cup C}$ is in~${G[A \setminus B]}$
    and every component of~${G-S}$ intersecting~${R' \cup C'}$ is in~${G[B \setminus A]}$, a contradiction.
\end{proof}

\subsection{Groups}

Let $\Gamma_i$ be a group for each $i\in [m]$.
We refer to the \emph{direct product} of these groups by~${\prod_{i \in [m]} \Gamma_i}$ and denote by~$\pi_j$ the projection map from~${\prod_{i \in [m]} \Gamma_i}$ to~$\Gamma_j$ for~${j \in [m]}$. 
For an element~$g$ of~$\prod_{i \in [m]} \Gamma_i$, we refer to the image~$\pi_i(g)$ as the \emph{$i$-th coordinate of~$g$}. 
When we say~${\Gamma = \prod_{i \in [m]} \Gamma_i}$ is a product of groups, we implicitly use this notation. 

For a non-empty set of elements~${S = \{a_i \colon i \in [t] \}}$ in a group~$\Gamma$, we denote by~${\gen{S}}$ or~${\gen{ a_i \colon i \in [t] }}$ the \emph{subgroup generated by~$S$}, which is the intersection of all subgroups of~$\Gamma$ containing~$S$.

\subsection{Group-labelled graphs}

Let~$\Gamma$ be an abelian group. 
A \emph{$\Gamma$-labelled graph} is a pair of a graph~$G$ and a function~${\gamma \colon E(G) \to \Gamma}$.
We say that~$\gamma$ is a \emph{$\Gamma$-labelling} of~$G$. 
A \emph{subgraph} of a $\Gamma$-labelled graph~${(G,\gamma)}$ is a $\Gamma$-labelled graph~$(H,\gamma')$ such that~$H$ is a subgraph of~$G$ and~$\gamma'$ is the restriction of~$\gamma$ to~$E(H)$. 
By a slight abuse of notation, we may refer to this $\Gamma$-labelled graph by~$(H,\gamma)$. 

For a $\Gamma$-labelled graph~${(G,\gamma)}$ and a subgraph~${H \subseteq G}$, we define~$\gamma(H)$ as~${\sum_{e \in E(H)} \gamma(e)}$, 
which we call the \emph{$\gamma$-value of~$H$}. 
Note that this definition implies that the $\gamma$-value of the empty subgraph is~$0$. 
We say that a subgraph~$H$ is \emph{$\gamma$-non-zero} if~${\gamma(H) \neq 0}$, and otherwise, we call it \emph{$\gamma$-zero}. 

We will often consider the special case where~$\Gamma$ is the product~${\prod_{i \in [m]} \Gamma_i}$ of~$m$ abelian groups for a positive integer~$m$. 
In this case, we denote by~$\gamma_i$ the composition of~$\gamma$ with the projection to~$\Gamma_i$. 

A $\Gamma$-labelled graph~${(G,\gamma)}$ is \emph{$\gamma$-bipartite} if every cycle of~$G$ is $\gamma$-zero. 

We frequently take a subgroup~$\Lambda$ of~$\Gamma$ and consider a new labelling using the quotient group~${\Gamma/\Lambda}$. 
For a $\Gamma$-labelled graph~${(G, \gamma)}$ and a subgroup~$\Lambda$ of~$\Gamma$, 
the~${\Gamma/\Lambda}$-labelling~$\lambda$ defined by~${\lambda(e) := \gamma(e) + \Lambda}$ for all edges~${e \in E(G)}$ is the \emph{induced $(\Gamma/\Lambda)$-labelling} of~$(G,\gamma)$. 

We will use the following duality theorem between packing and covering of $\gamma$-non-zero $A$-paths. 

\begin{theorem}[Wollan~\cite{Wollan2010}]
    \label{thm:tpath}
    Let~$k$ be a positive integer, let~$\Gamma$ be an abelian group, let~$(G,\gamma)$ be a $\Gamma$-labelled graph, and let~${A \subseteq V(G)}$. 
    Then~$G$ contains~$k$ vertex-disjoint $\gamma$-non-zero $A$-paths or 
    there exists a set~${X \subseteq V(G)}$ of size at most~${f_{\ref*{thm:tpath}}(k) := 50k^4}$ such that~${G-X}$ has no $\gamma$-non-zero $A$-paths.
\end{theorem}

Let~$x$ be a vertex of~$G$ and let~${\delta \in \Gamma}$ be an element of order~$2$. 
For each edge~$e$ of~$G$, let 
\[
    \gamma'(e)=
    \begin{cases}
        \gamma(e) + \delta & \text{if $e$ is incident with $x$,}\\
        \gamma(e) &\text{otherwise.}
    \end{cases}
\]
We say that~$\gamma'$ is obtained from~$\gamma$ by \emph{shifting by~$\delta$ at~$x$}. 
Observe that this shift does not change the weight sum of a cycle because~${\delta + \delta = 0}$.
We say two $\Gamma$-labellings~$\gamma_1$ and~$\gamma_2$ of~$G$ are \emph{shifting-equivalent} 
if~$\gamma_1$ can be obtained from~$\gamma_2$ by a sequence of shifting operations. 

The following lemma asserts that for $\gamma$-bipartite graphs we can find a shifting-equivalent $\Gamma$-labelling~$\gamma'$ in which every corridor is $\gamma'$-zero. 
Similar ideas appear in Geelen and Gerards~\cite{GeelenG2009}.

\begin{lemma}
    \label{lem:shifting}
    Let~$\Gamma$ be an abelian group, let~${(G,\gamma)}$ be a~$\Gamma$-labelled graph and let~${H \subseteq G}$ be a subdivision of a $3$-connected 
    graph~${\hat{H}}$. 
    If~$H$ is $\gamma$-bipartite, then~$\gamma$ is shifting-equivalent to a $\Gamma$-labelling~$\gamma'$ such that every corridor of~$H$ is $\gamma'$-zero.
\end{lemma}

\begin{proof}
    Let~$T$ be a spanning tree of~${\hat{H}}$ rooted at some~${r \in V({\hat{H}})}$. 
    It is enough to find a $\Gamma$-labelling~$\gamma'$ of~$G$ which is shifting-equivalent to~$\gamma$, such that all corridors of~$H$ corresponding to edges in~$T$ are $\gamma'$-zero, 
    because~$H$ is $\gamma'$-bipartite. 
    Choose a $\Gamma$-labelling~$\gamma'$ shifting-equivalent to~$\gamma$ and a subtree~$T'$ of~$T$ containing~$r$ such that all corridors of~$H$ corresponding to edges in~$T'$ are $\gamma'$-zero, and subject to these conditions,~$\abs{V(T')}$ is maximised.
    
    Suppose that~${T' \neq T}$. 
    Then there is an edge~$vw$ of~$T$ such that~${v \in V(T')}$ and~${w \notin V(T')}$.
    Let~$P$ be the corridor of~$H$ corresponding to the edge~$vw$.
    Since~$\hat{H}$ is $3$-connected, 
    there is a cycle~$O$ in~${H-E(P)}$ containing~$v$ and~$w$. 
    Let~$O_1$ and~$O_2$ denote the distinct cycles in~${O \cup P}$ containing~$P$.
    Since $H$ is $\gamma'$-bipartite, we have that ${\gamma'(P) + \gamma'(P) = \gamma'(O_1) + \gamma'(O_2) - \gamma'(O) = 0}$, 
    and hence~$\gamma'(P)$ is an element of order at most~$2$. 
    Let~$\gamma''$ be a $\Gamma$-labelling of~$G$ obtained from~$\gamma'$ 
    by shifting by~$\gamma'(P)$ at~$w$. 
    Let~${T'' = T[V(T') \cup \{w\}]}$. 
    Then all corridors of~$H$ corresponding to edges of~$T''$ are $\gamma''$-zero, contradicting the choice of~$\gamma'$ and~$T'$.
\end{proof}

\section{Discussion}
\label{sec:overview}

\subsection{Proof Sketch}
\label{subsec:sketch}

We now sketch the proof of Theorem~\ref{thm:main}, which will proceed by induction on~$k$.
We consider the group~${\Gamma := \prod_{i \in [m]} \Gamma_i}$ and a single $\Gamma$-labelling, which simplifies the arguments we present and in particular allows us to consider quotient groups. 
The goal will be to show that if the smallest hitting set~$T$ for the cycles in~$\mathcal{O}$ is sufficiently larger than~${f_m(k-1)}$, then there is a half-integral packing of~$k$ cycles in~$\mathcal{O}$. 
To construct this packing, in Section~\ref{sec:pack}, we first find a tangle whose order is correlated with~${\abs{T}/f_m(k-1)}$, which allows us to construct a large wall as a subgraph in~$G$. 
In particular, this wall will have the property that no cycle in~$\mathcal{O}$ can be separated from the nails of the wall by deleting a small set of vertices. 
Our strategy will be to find disjoint sets of $\gamma$-non-zero paths and use the structure of the wall to connect them up to form cycles. 
But before we begin to do this, in Section~\ref{sec:cleanwalls} we find a subwall~$W$ of the original wall such that the $N^{W}$-paths in~$W$ have some nice homogeneity properties with respect to the group labelling. 
It would be simplest if we could guarantee that all $N^{W}$-paths in~$W$ were $\gamma$-zero, but this is not feasible with arbitrary abelian groups. 
Instead, we deal separately with the factors~$\Gamma_i$ of~$\Gamma$ for which we can guarantee that all $N^{W}$-paths in~$W$ are $\gamma_i$-zero, and the factors of~$\Gamma$ for which we cannot find any large subwall of~$W$ with this property.

Applying Theorem~\ref{thm:tpath}, we can find for every factor~$\Gamma_i$ of~$\Gamma$ a large set of disjoint~$N^{W}$-paths which are $\gamma_i$-non-zero. 
The difficulty here lies in combining these paths together such that for all~${i \in [m]}$ the total $\gamma_i$-value is not in~$\Omega_i$. 
To achieve this, in Section~\ref{sec:handles} we show how to iteratively find sets of disjoint paths which are non-zero with respect to a quotient group defined in terms of the previously constructed paths, and link them up to the boundary of a subwall of~$W$. 
In Section~\ref{sec:abelian}, we analyse conditions under which we can find a set of elements from a product of abelian groups whose sum avoids a finite set~$\Omega_i$ in each coordinate. 
In Section~\ref{sec:handles2cycles}, we discuss how to combine the disjoint sets of paths into cycles using the wall. 
This allows us to find a half-integral packing of~$k$ cycles whose $\gamma_i$-values are in~${\Gamma_i \setminus \Omega_i}$ for every factor~$\Gamma_i$ of~$\Gamma$ for which the $N^{W}$-paths in~$W$ are $\gamma_i$-zero. 

To deal with each remaining factor~$\Gamma_i$, we observe that since no large subwall of~$W$ is $\gamma_i$-bipartite, every large subwall of~$W$ contains a $\gamma_i$-non-zero cycle. 
In fact, we iteratively find disjoint cycles in~$W$ which are non-zero with respect to quotient groups defined in terms of the previously constructed cycles. 
We link these cycles up to the half-integral packing of $k$-cycles we have constructed, and by rerouting through them, transform each cycle in our half-integral packing into a cycle in~$\mathcal{O}$. 

\subsection{%
\texorpdfstring{Obstructions for integral and half-integral Erd\H{o}s-P\'{o}sa type results}%
{Obstructions for integral and half-integral Erdos-Posa type results}}
\label{subsec:obstructionsEP}

One fundamental obstruction to Erd\H{o}s-P\'{o}sa type results, known as the Escher wall, is due to Lov\'{a}sz and Schrijver~(see~\cite{Thomassen1988}). 
An \emph{Escher wall of height~$n$} is obtained from an ${(n, n)}$-wall~$W$ by adding a family~${(P_i \colon i \in [n])}$ of vertex-disjoint $W$-paths, such that for each~${i \in [n]}$, one endvertex of~$P_i$ is in~${R^W_1 \cap C^W_i}$ and the other is in~${R^W_n \cap C^W_{n+1-i}}$. See Figure~\ref{fig:escher} for an illustration.
They observed that there are no two disjoint cycles in such an Escher wall which each contain an odd number of paths in~${(P_i \colon i \in [n])}$, 
but for any set~$S$ of at most~${n-1}$ vertices, there is a cycle of the Escher Wall containing exactly one path in~${(P_i \colon i \in [n])}$.
Using this construction, Thomassen~\cite{Thomassen1988} argued that an analogue of the Erd\H{o}s-P\'{o}sa theorem does not hold for odd cycles. 
It can be further used to show that the same holds for cycles of length $\ell$ modulo $m$ whenever $m$ is an even positive integer and $\ell$ is an odd integer with $0<\ell<m$, see Wollan~\cite{Wollan2011}. 

\begin{figure}
    \centering
    \begin{tikzpicture}[scale=.6]
        \tikzstyle{w}=[circle,draw,fill=black!50,inner sep=0pt,minimum width=3pt]
        \useasboundingbox (0,-1) rectangle (7.5,4);
        \foreach \x in {0,...,13}{
            \ifthenelse{\x>0}{
                \pgfmathtruncatemacro{\start}{0};
            }{
                \pgfmathtruncatemacro{\start}{1};
            }
            \ifthenelse{\x<13}{
                \pgfmathtruncatemacro{\t}{6};
            }{
                \pgfmathtruncatemacro{\t}{5};
            }
            \foreach \y in {\start,...,\t} {
                \node at (\x/2,\y/2) [w] (v\x\y) {};
            }
        }        
        \foreach \x in {0,1,...,13}{
            \ifthenelse{\x>0}{
                \pgfmathtruncatemacro{\start}{0};
            }{
                \pgfmathtruncatemacro{\start}{1};
            }
            \ifthenelse{\x<13}{
                \pgfmathtruncatemacro{\t}{6};
            }{
                \pgfmathtruncatemacro{\t}{5};
            }          
            \foreach \y in {\start,...,\t} {
                \pgfmathtruncatemacro{\nextx}{\x+1}
                \ifthenelse{\x<12}{ 
                    \draw [gray](v\x\y)--(v\nextx\y);
                }{}
                \ifthenelse{\x=12 \and \y<6}{ 
                    \draw [gray](v\x\y)--(v\nextx\y);
                }{}                
                \pgfmathtruncatemacro{\nexty}{\y+1}
                \pgfmathtruncatemacro{\sum}{\x+\y}
                \ifthenelse{\isodd{\sum}\and \y<6}{
                    \draw [gray](v\x\y)--(v\x\nexty);
                }{}
            }
        }
        \draw [out=20](v06) .. controls (11,6) and (11,-3) .. (v120);
        \draw [out=20](v26) .. controls (11,6) and (11,-3) .. (v100);
        \draw [out=20](v56).. controls (11,6) and (11,-3) ..(v80);
        \draw [out=20](v66) .. controls (11,6) and (11,-3) .. (v70);
        \draw [out=20](v96) .. controls (11,6) and (11,-3) .. (v40);
        \draw [out=20](v116).. controls (11,6) and (11,-3) .. (v20);
        \draw [out=20](v126) .. controls (11,6) and (11,-3) .. (v10);
    \end{tikzpicture}
    \caption{An Escher wall of height~$7$.}
    \label{fig:escher}
\end{figure}
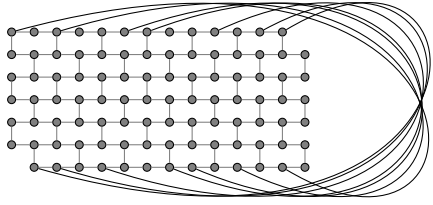

Many previous half-integral Erd\H{o}s-P\'{o}sa type results have relied on characterising Escher walls as the fundamental obstructions to finding integral packings of certain classes of cycles. 
Half-integral packings are then obtained using the structure of the Escher Wall. 
In our setting, the Escher Wall is not the only possible obstruction which can arise. 
In fact, unlike Escher walls, the following type of obstruction can occur even in planar graphs. 
For this reason, some of the structural results which have been used to obtain other half-integral Erd\H{o}s-P\'{o}sa type results are unlikely to be useful in our setting. 
For example, we do not use the flat wall theorem of Robertson and Seymour~\cite{RobertsonS95, RobertsonS1986} in this paper, and there is no obvious way to significantly simplify our proofs by doing so. 

\begin{proposition}
    \label{prop:obstruction2}
	Let~$G$ be the ${n \times n}$-grid, and let~$\mathcal{O}$ be the set of cycles of~$G$ 
	which contain at least one edge of the top row~$R_t$ of~$G$, 
	at least one edge of the bottom row~$R_b$ of~$G$, 
	and at least one edge from the leftmost column~$C_{\ell}$ of~$G$, 
	and let~$\mathcal{O}'$ be the set of cycles of~$G$ 
	which contain exactly one edge of the top row~$R_t$ of~$G$, 
	exactly one edge of the bottom row~$R_b$ of~$G$, 
	and exactly one edge from the leftmost column~$C_{\ell}$ of~$G$. 
	Then every pair of cycles in~$\mathcal{O}$ intersect, but there is no hitting set for~$\mathcal{O}'$ of size less than~${(n-1)/2}$.
\end{proposition}

\begin{proof}
	We may assume that~${n \geq 3}$. 
	Note that the graph~$G'$ obtained from~$G$ by adding a vertex~$z$ adjacent to every vertex of~${R_t \cup R_b \cup C_{\ell}}$ is planar. 
	Suppose for contradiction that there are disjoint cycles~$O_1$ and~$O_2$ in~$\mathcal{O}$, 
	and let~$H$ be the component of~${G-O_1}$ containing~$O_2$. 
	Since~$C_{\ell}$ is connected~$O_1$ intersects~${C_{\ell} - (R_t \cup R_b)}$, there is a vertex~$v_1$ in~${(O_1 \cap C_{\ell}) - (R_t \cup R_b)}$ adjacent to a vertex of~$H$, 
	and similarly a vertex~$v_2$ in~${O_1 \cap R_t}$ adjacent to a vertex of~$H$ and a vertex~$v_2$ in~${O_1 \cap R_t}$ adjacent to a vertex of~$H$. 
	Contracting~$C_1$ to a triangle on~${\{v_1,v_2,v_3\}}$ and~$H$ to a single vertex, we find a~$K_5$-minor in~$G'$, contradicting Wagner's Theorem. 
	
	Now consider~${S \subseteq V(G)}$ of size less than~${(n-1)/2}$. 
	Note that there are two adjacent columns~$C_i$ and~$C_{i+1}$ which do not intersect~$S$, 
	and likewise two adjacent rows~$R_j$ and~$R_{j+1}$ which do not intersect~$S$. 
	It is easy to see that the subgraph of~$G$ induced on the vertices in~${C_i \cup C_{i+1} \cup R_j \cup R_{j+1}}$ contains a cycle in~$\mathcal{O}'$. 
\end{proof}

As an example, consider an~${n \times n}$-grid in which all edges on the top row are subdivided exactly~$14$ times, 
all edges of the bottom row are subdivided exactly~$69$ times, 
all edges of the leftmost column are subdivided exactly~$20$ times, 
and all other edges are subdivided exactly~$104$ times. 
By Proposition~\ref{prop:obstruction2}, there are no two vertex-disjoint cycles of length~${1 \mod 105}$, 
and no hitting set for these cycles of size less than~${(n-1)/2}$. 

Another consequence of Proposition~\ref{prop:obstruction2} is that an analogue of the Erd\H{o}s-P\'osa theorem does not hold for cycles intersecting three prescribed vertex sets. 
To see this, consider an~${n \times n}$-grid, 
let~$S_1$ be the set of all vertices on the top row, 
let~$S_2$ be the set of all vertices of the bottom row, 
and let~$S_3$ be the set of all vertices of the leftmost column. 
By Proposition~\ref{prop:obstruction2}, there are no two vertex-disjoint cycles each containing at least one vertex from each of~$S_1$,~$S_2$,~$S_3$, and no hitting set for these cycles of size less than~${(n-1)/2}$.

\medskip

Our main theorem demonstrates that it is not easy to find non-trivial obstructions for half-integral Erd\H{o}s-P\'{o}sa type results. 
However, the following proposition does allow us to describe settings in which even a half-integral Erd\H{o}s-P\'{o}sa type result is not possible. 

\begin{proposition}
    \label{prop:counterexample-for-half-integral}
	Let~$t$ and~$c$ be positive integers, let~$\Gamma$ be an abelian group and let~$\Omega$ be a subset of~$\Gamma$ such that there is an element~${g \in \Gamma}$ and an integer~${d > (c-1)t}$ such that~$d$ is the minimum integer greater than~$2$ for which~${dg \notin \Omega}$. 
	Then there is a graph $G$ with $\Gamma$-labelling $\gamma$ such that, for the set~$\mathcal{O}$ of all cycles of~$G$ whose $\gamma$-values are not in~${\Omega}$, every~$c$ cycles in~$\mathcal{O}$ share a common vertex, but there is no hitting set for~$\mathcal{O}$ of size less than~$t$. 
\end{proposition}

\begin{proof}
    The~${c = 1}$ case is trivial, so we may assume~${c \geq 2}$. 
    Let~${n := \lceil cd/(c-1) \rceil - 1}$, 
    let~${G := K_n}$, 
    and let~$\gamma$ be the $\Gamma$-labelling assigning~$g$ to every edge of~$G$. 
    By construction, every cycle in~$\mathcal{O}$ has length greater than~${(c-1)n/c}$, and so every~$c$ cycles in~$\mathcal{O}$ share a common vertex. 
    However, every cycle of length~$d$ in~$G$ is in~$\mathcal{O}$, so the smallest hitting set for~$\mathcal{O}$ has size~${n-(d-1) > t}$. 
\end{proof}

As an example, consider an infinite abelian group which contains arbitrarily large finite cyclic subgroups. 
One consequence of Proposition~\ref{prop:counterexample-for-half-integral} is that there is no half-integral Erd\H{o}s-P\'{o}sa type result for cycles of weight zero in graphs labelled with such a group. 

\medskip

As another consequence, we obtain a lower bound on the functions mentioned in Theorem~\ref{thm:main} which depends on both~$m$ and~$\omega$. 

\begin{corollary}
	For any function~${f_{m, \omega}}$ as in Theorem~\ref{thm:main}, we have~${f_{m,\omega}(3) > (2+ m \omega )/ 2}$. 
\end{corollary}

\begin{proof}
    Consider for each~${i \in [m]}$ the group~${\Gamma_i := \mathbb{Z}}$ and the subset~${\Omega_i := [2+\omega i] \setminus [2+\omega (i-1)]}$. 
    The result follows 
    with~${c = 3}$,~${g = ( 1 \colon i \in [m])}$, 
    and~${\Omega = \bigcup_{i\in [m]} \{ g' \in \Gamma \colon \pi_i(g') \in \Omega_i \}}$ from Proposition~\ref{prop:counterexample-for-half-integral}. 
\end{proof}

\subsection{Relating vertex-labellings to edge-labellings}
\label{subsec:vxlabel}

We now demonstrate how to convert a group labelling of the vertices of a graph to a group labelling of its edges. 
Corollary~\ref{cor:vxlabelling} follows easily from Theorem~\ref{thm:main} after applying the following lemma to each of the vertex-labellings it mentions. 

\begin{lemma}
    Let~$\Gamma$ be an abelian group, let~$\Omega \subseteq \Gamma$ be a finite subset, let~$G$ be a graph, and let~$\gamma$ be a $\Gamma$-vertex-labelling of~$G$. 
    Then there is a group~$\Gamma'$, a subset~$\Omega' \subseteq \Gamma'$ with~${\abs{\Omega'} = \abs{\Omega}}$, and a $\Gamma'$-labelling~$\gamma'$ of~$G$ such that for every cycle~$O$ of~$G$, we have 
    ${\gamma(O) \in \Omega}$ if and only if~${\gamma'(O) \in \Omega'}$. 
\end{lemma}

\begin{proof}
    Let~${\Gamma'' = \gen{ \Omega \cup \{\gamma(v) \colon v \in V(G)\}}}$. 
    By the fundamental theorem of finitely generated abelian groups, there exist an integer~$m$ and an isomorphism~$\varphi$ from~$\Gamma''$ to a product~${\prod_{i \in [m]} \Gamma_i}$, where each~$\Gamma_i$ is either~$\mathbb{Z}$, or a cyclic group~$\mathbb{Z}_{n_i}$ of order~$n_i$. 
    For~${i \in [m]}$, let~${\Gamma_i' := \mathbb{Z}_{2n_i}}$ if the order of~$\Gamma_i$ is finite, and~${\Gamma_i' := \Gamma_i}$ otherwise. 
    Let~${\Gamma' := \prod_{i \in [m]} \Gamma_i'}$. 
    For~${i \in [m]}$, let~$e_i$ denote the element of~$\Gamma''$ such that~${\pi_j(\varphi(e_i)) = 1}$ if~${i = j}$, and~${\pi_j(\varphi(e_i)) = 0}$ if~${i \neq j}$. 
    Then~$\Gamma'' = \gen{ e_i \colon i \in [r']}$. 
    For~${i \in [m]}$, let~$e_i'$ denote the element of~$\Gamma'$ such that~${\pi_j(e_i') = 1}$ if~${i = j}$, and~${\pi_j(e_i') = 0}$ if~${i \neq j}$. 
    For each~${j \in [2]}$, we define a homomorphism~$\psi_j$ from~$\Gamma''$ to~$\Gamma'$ by setting~${\psi_j(e_i) := j e_i'}$ on the generators. 
    Note that~$\psi_2$ is injective since the kernel of~$\psi_2$ is trivial, 
    that the image of~$\psi_2$ is~${2 \Gamma'}$, 
    and that~${\psi_1(2g) = \psi_2(g)}$ for all~${g \in \Gamma''}$. 
    We define~${\Omega' := \psi_2(\Omega)}$ and a $\Gamma'$-labelling~$\gamma'$ of~$G$ for an edge~$e = vw$ of~$G$ by setting~${\gamma'(e) = \psi_1\big(\gamma(v) + \gamma(w)\big)}$. 
    Note that for a cycle~$O$ of~$G$, we have that
    \[
        \gamma'(O) 
        = \sum_{e \in E(O)} \gamma'(e) 
        = \sum_{vw \in E(O)} \psi_1( \gamma(v) + \gamma(w) )
        = \psi_1 \left( 2 \sum_{v \in V(O)} \gamma(v) \right)
        = \psi_2 ( \gamma(O) ). 
    \]
    Hence, the result follows from the injectivity of~$\psi_2$. 
\end{proof}

\subsection{Graphs embedded on a surface}
\label{subsec:surface}

We now discuss how our result applies to graphs embedded on a surface, where we consider the first homology group with coefficients in~$\mathbb{Z}_2$. 
Huynh, Joos, and Wollan~\cite[Proposition 5]{HuynhJW2017} demonstrated that given a graph~$G$ embedded on a surface whose $\mathbb{Z}$-homology group is~$\Gamma$, there is a directed $\Gamma$-labelling of~$G$ so that the set of cycles in~$G$ that are homologous to zero is exactly the set of cycles having group value~$0$ in the labelling. 
This allowed them to obtain for graphs embedded on a surface a half-integral Erd\H{o}s-P\'{o}sa result for the non-null-homologous cycles of the embedding. 
Our result works in essentially the same way. 

A graph~$H$ is called \emph{even} if every vertex of~$H$ has even degree. 
For a graph~$G$, let~$\mathcal{C}(G)$ denote the \emph{cycle space} of~$G$ over~$\mathbb{Z}_2$, that is the vector space of all even subgraphs~$H$ of~$G$ with the symmetric difference as the operation. 

\begin{proposition}
    \label{prop:cyclespace-hom-to-gamma}
    Let~$G$ be a graph, let~$\Gamma$ be an abelian group, and let~${\phi \colon \mathcal{C}(G) \to \Gamma}$ be a group homomorphism. 
    Then there is a $\Gamma$-labelling~$\gamma$ of~$G$ such that~${\gamma(H) = \phi(H)}$ for every even subgraph~$H$ of~$G$. 
\end{proposition}

\begin{proof}
    Without loss of generality, we may assume that~$G$ is connected. 
    Let~$T$ be a spanning tree of~$G$. 
    For each edge~${e \in E(G) \setminus E(T)}$, let~$C_{e,T}$ denote the unique cycle in~${T + e}$. 
    We define~${\gamma(e) := 0}$ for each~${e \in E(T)}$ and~${\gamma(e) := \phi(C_{e,T})}$ for each~${e \in E(G) \setminus E(T)}$. 
    The statement now trivially follows, because the set~$\{C_{e,T} \colon e \in E(G) \setminus E(T)\}$ forms a basis of the cycle space (see~\cite[Theorem~1.9.5]{Diestel5}). 
\end{proof}

Now for a graph~$G$ embedded in a surface~$\Sigma$, the map assigning each even subgraph its $\mathbb{Z}_2$-homology class is a group homomorphism from~$\mathcal{C}(G)$ to the $\mathbb{Z}_2$-homology group of~$\Sigma$. 
Hence, Corollary~\ref{cor:surface} follows with Proposition~\ref{prop:cyclespace-hom-to-gamma} from Theorem~\ref{thm:main}. 

Note that for a closed orientable surface, the set of simple closed curves homologous to zero for the $\mathbb{Z}_2$-homology is exactly the same as for the $\mathbb{Z}$-homology. 
This follows the universal coefficient theorem (see \cite{Hatcher02}), which allows us to relate the $\mathbb{Z}$-homology with the $\mathbb{Z}_2$-homology by taking all coefficients modulo~${2}$. 
We then apply a classical result which states that no simple closed curve has $\mathbb{Z}$-homology class~${kh}$ for any integer~${k \geq 2}$ and any non-zero element~$h$ of the $\mathbb{Z}$-homology (see for example \cite{Schafer1976}). 
Hence, in the case of graphs embedded on closed orientable surfaces, we recover the result of Huynh, Joos and Wollan for non-null-homologous cycles.

\section{Packing functions and hitting sets}
\label{sec:pack}

In this section, we introduce the concept of packing functions as a tool to generalise the ideas of both integral and half-integral packings of subgraphs, 
which enables us to discuss these and similar ideas in a unified way. 

For a function~$\nu$ from the set of subgraphs of a graph~$G$ to the set of non-negative integers, we say that 
\begin{itemize}
    \item $\nu$ is \emph{monotone} if~$\nu(H) \leq \nu(H')$ whenever~$H$ is a subgraph of~$H'$, 
    \item $\nu$ is \emph{additive} if~$\nu(H \cup H') = \nu(H) + \nu(H')$ whenever~$H$ and~$H'$ are vertex-disjoint subgraphs of $G$, and 
    \item $\nu$ is a \emph{packing function for~$G$} if it is monotone and additive.
\end{itemize}
Now let~${\nu}$ be a packing function for a graph~$G$.  
For a subgraph~${H \subseteq G}$, we say a set~${T \subseteq V(H)}$ is an \emph{$\nu$-hitting set for~$H$} if~${\nu(H - T) = 0}$. 
We define~$\tau_\nu(H)$ as the size of a smallest $\nu$-hitting set of~$H$. 
Note that in the traditional sense of the word, a $\nu$-hitting set of~$G$ is a hitting set for the minimal subgraphs~${H \subseteq G}$ for which~${\nu(H) \geq 1}$.  

For example, a function mapping a subgraph~$H$ of~$G$ to the maximum number of vertex-disjoint cycles in~$H$ is a packing function of~$G$. 

\medskip

The following lemma argues that if~${\nu(G)}$ is small but~$G$ has no small $\nu$-hitting set, then every minimum $\nu$-hitting set 
induces a tangle of large order. 
Similar arguments for specific packing functions appear many times in the literature, 
see~\cite{HuynhJW2017} and~\cite{ReedRST1996} for instance.  

\begin{lemma}
    \label{lem:welllinked}
    Let~$\nu$ be a packing function for a graph~$G$ 
    and let~${T \subseteq V(G)}$ be a minimum $\nu$-hitting set for~$G$ of size~${t}$. 
    Let~$\mathcal{T}_T$ be the set of all separations~${(A,B)}$ of~$G$
    of order less than~${t/6}$ such that~${\abs{B \cap T} > 5t/6}$. 
    If~${\tau_\nu(H) \leq t/12}$ whenever~$H$ is a subgraph of~$G$ with~${\nu(H) < \nu(G)}$, 
    then~$\mathcal{T}_T$ is a tangle of order~${\lceil t/6 \rceil}$. 
\end{lemma}

\begin{proof}
    First, we show the following claim. 
    
    \begin{claim*}
        Let~${X, Y \subseteq T}$ be disjoint sets with~${\abs{X} = \abs{Y} \geq t/6}$. 
        Then there is a linkage in~$G$ from~$X$ to~$Y$ of order~$\abs{X}$ containing no vertex in~${Z := T \setminus (X \cup Y)}$. 
    \end{claim*}
    
    \begin{proofofclaim}
        Suppose for a contradiction that there is no such linkage. 
        By Menger's theorem applied to~${G - Z}$, 
        there is a separation~${(A, B)}$ of~$G$ of order strictly less than~${\abs{X} + \abs{Z}}$ with~${Z \subseteq A \cap B}$, ${X \subseteq A}$, and ${Y \subseteq B}$. 
        Let~${S := A \cap B}$.
        Now observe that 
        \[
            {\nu(G - (A \cup T))+\nu(G - (B \cup T)) = \nu(G - (S \cup T)) \leq \nu(G - T) = 0}, 
        \] 
        and so~${\nu(G - (A \cup T)) = \nu(G - (B \cup T)) = 0}$. 
        Hence 
        \[
            {\nu(G - B) = \nu(G - B) + \nu(G - (A \cup T)) = \nu(G - B) + \nu(G - (A \cup Y)) = \nu(G - (S \cup Y))},
        \] 
        and so by the minimality of~$T$ and the fact that~${\abs{S \cup Y} \leq \abs{S} + \abs{Y} < \abs{X} + \abs{Z} + \abs{Y} = \abs{T}}$ we have that~${\nu(G - B) \geq 1}$. 
        By symmetry~${\nu(G - A) \geq 1}$, 
        and since~${\nu(G - A) + \nu(G - B) \leq \nu(G)}$, both~${\nu(G - A)}$ and~${\nu(G - B)}$ are strictly less than~${\nu(G)}$. 

        By the assumption, ${\tau_\nu(G - A), \tau_\nu(G - B) \leq t/12}$.
        Let~$T_{A}$ and~$T_{B}$ be $\nu$-hitting sets of minimum size for~${G-A}$ and~${G-B}$, respectively. 
        Then
        \[
            {\nu(G - (T_A \cup T_B \cup S)) = \nu(G - (T_{A} \cup A)) + \nu(G - (T_{A} \cup B)) = 0},
        \] 
        but
        ${\abs{T_A} + \abs{T_B} + \abs{S}
            \leq (t/6) + \abs{S}
            \leq \abs{Y} + \abs{S} 
            < \abs{T}}$, contradicting the assumption that~$T$ is a minimum $\nu$-hitting set.
    \end{proofofclaim}
    
    Let~${(A,B)}$ be a separation of order less than~${t/6}$ with~${\abs{B \cap T} \geq \abs{A \cap T}}$, and let~${S := A \cap B}$. 
    Clearly,~${(B,A) \notin \mathcal{T}_T}$. 
    Suppose for a contradiction that~${(A,B) \notin \mathcal{T}_T}$ and hence~${A \setminus B}$ contains at least~${t/6}$ vertices of~$T$. 
    Since~${B \setminus A}$ contains at least as many vertices of~$T$ as~${A \setminus B}$ does, by the claim there is a linkage of size~${\lceil t/6 \rceil}$ in~$G$ from~${A \setminus B}$ to~${B \setminus A}$, contradicting the assumption on the order of~${(A,B)}$.
    Hence, ${(A,B) \in \mathcal{T}_T}$.
    
    Note that~${\abs{T \cap A} < t/3}$ for each~${(A,B) \in \mathcal{T}_T}$. 
    Hence for~${(A_1,B_2), (A_2,B_2), (A_3,B_3) \in \mathcal{T}_T}$ we have that~${\abs{T \cap (A_1 \cup A_2 \cup A_3)} < t}$, and hence~${G[A_1] \cup G[A_2] \cup G[A_3] \neq G}$. 
    Thus we conclude that~$\mathcal{T}_T$ is a tangle of order~${\lceil t/6\rceil}$.
\end{proof}

Let us now turn our attention to packing functions~$\nu$ for a $\Gamma$-labelled graph~${(G,\gamma)}$ for an abelian group~$\Gamma$. 
The following lemma is useful for converting between $\gamma$-non-zero cycles and $\gamma$-non-zero paths.
We will appeal to it in the final lemma of this section, and again in Lemma~\ref{lem:breaking}.

\begin{lemma}
    \label{lem:non-zero-path-in-cycle}
    Let~$\Gamma$ be an abelian group, 
    let~${(G, \gamma)}$ be a $\Gamma$-labelled graph, 
    let~$O$ be a $\gamma$-non-zero cycle in~$G$, and 
    let~${T \subseteq V(G)}$.
    If~$G$ contains three vertex-disjoint ${(V(O),T)}$-paths~$P_1$, $P_2$, $P_3$, 
    then~${H := O \cup P_1\cup P_2\cup P_3}$ contains a $\gamma$-non-zero $T$-path. 
\end{lemma}
 
\begin{proof}
    We may assume that~${\abs{V(P_i) \cap V(O)} = \abs{V(P_i) \cap T} = 1}$ for all~${i \in [3]}$ by taking a subpath if necessary. 
    For~${i \in [3]}$, 
    let~$Q_i$ and~$Q'_i$ be the two paths in~$H$ each having~${\bigcup_{j \in [3] \setminus \{i\}} V(P_j) \cap T}$ as its set of endvertices, 
    where~$Q'_i$ is the path that is disjoint from~$P_i$. 
    Now, 
    \[
        \sum \limits_{i \in [3]} \left( \gamma(Q_{i}) - \gamma(Q'_{i}) \right) 
        = 2 \cdot \sum \limits_{i \in [3]} \left (\gamma(P_i) - \gamma(P_i) \right) + 2 \cdot \gamma(O) - \gamma(O) 
        = \gamma(O)
        \neq 0.
    \]
    Hence, for some~${i \in [3]}$, one of the paths~$Q_{i}$ or~$Q'_{i}$ is~$\gamma$-non-zero. 
    And since this path is the edge-disjoint union of $T$-paths, it contains a~$\gamma$-non-zero $T$-path, as desired. 
\end{proof}

Given an abelian group~$\Gamma$ and a $\Gamma$-labelled graph ${(G,\gamma)}$, we are interested in the packing function~$\nu$ for~$G$ which maps a subgraph~$H$ of~$G$ to the size of the largest half-integral packing 
of the type of cycles of~$H$ we are considering. 
To prove our main result, we will need the following tool for finding disjoint sets of paths in~$G$ which are non-zero with respect to the induced labelling of certain quotient groups, which we will later construct.

\begin{lemma}
    \label{lem:cover}
    Let~$u$, $k$ be positive integers such that~${f_{\ref*{thm:tpath}}(k) < u - 2}$.
    Let~$\Gamma$ be an abelian group, 
    let~${(G, \gamma)}$ be a $\Gamma$-labelled graph, and let~$\nu$ be a packing function for~$G$ such that
    \begin{itemize}
        \item every minimal subgraph~${H}$ of~$G$ with~${\nu(H) \geq 1}$ is a $\gamma$-non-zero cycle,
        \item ${\tau_\nu(H) \leq 3 u}$ for every subgraph~${H}$ of $G$ with~${\nu(H) < \nu(G)}$, and 
        \item ${\tau_\nu(G) \geq u}$.
    \end{itemize} 
    Let~${T \subseteq V(G)}$ be a minimum $\nu$-hitting set for~$G$ and let~${N \subseteq V(G)}$ such that for every~${S \subseteq V(G)}$ of size less than~${u}$, there is a component of~${G-S}$ containing a vertex of~$N$ and at least~$4u$ vertices of~$T$. 
    Then~$G$ contains~$k$ vertex-disjoint $\gamma$-non-zero $N$-paths.
\end{lemma}

\begin{proof}
    Suppose that $G$ does not contain~$k$ vertex-disjoint $\gamma$-non-zero $N$-paths. 
    By Theorem~\ref{thm:tpath}, there exists~${S \subseteq V(G)}$ of size less than~${u-2}$ hitting all $\gamma$-non-zero $N$-paths. 
    Since~${\abs{S} < u \leq \tau_\nu(G)}$, we have that~${\nu(G - S) \geq 1}$, so $G-S$ has a $\gamma$-non-zero cycle~$O$ with~${\nu(O) \geq 1}$. 
    By Lemma~\ref{lem:non-zero-path-in-cycle}, $G-S$ does not have three vertex-disjoint ${(V(O),N)}$-paths.
    By Menger's theorem applied to~${G-S}$, there exists~${S'\subseteq V(G)}$ of size at most~${\abs{S}+2}$ separating~$O$ from~$N$. 
    Since~${\abs{S'} < u}$, by the given assumption on~$N$, the graph $G-S'$ has a component~$H$ containing a vertex of~$N$ and at least~$4u$ vertices of~$T$.
    Now~${\nu(H) \leq \nu(H\cup O) - \nu(O) < \nu(H\cup O) \leq \nu(G)}$, so there is a $\nu$-hitting set~$T_H$ for~$H$ of size at most~$3u$. 
    Let~${T' := T_H \cup S' \cup (T \setminus V(H))}$, and observe that~${\abs{T'} \leq \abs{T_H} + \abs{S'} + \abs{T} - 4u < \abs{T}}$. 
    But~${\nu(G-T') \leq \nu(G - (S' \cup V(H) \cup T)) + \nu(H-T_H) = 0}$, 
    contradicting the assumption that~$T$ is a minimum $\nu$-hitting set.
\end{proof}

\section{Clean walls}
\label{sec:cleanwalls}

In the proof of our main theorem in Section~\ref{sec:main}, we will apply Theorem~\ref{thm:wall} and Lemma~\ref{lem:welllinked} to construct a wall~$W$ in a group-labelled graph. 
However, it will be useful to move to a large subwall of~$W$ which has some nice homogeneity properties. 
For this purpose, we introduce the following notion of cleanness.

Let ${\Gamma = \prod_{i \in [m]} \Gamma_i}$ be a product of~$m$ abelian groups 
and let~$(G,\gamma)$ be a $\Gamma$-labelled graph. 
Given a subset~${Z \subseteq [m]}$ and an integer~$\ell$, 
we say that a wall~$W$ in~$G$ is \emph{$(\gamma,Z,\ell)$-clean} if 

\begin{enumerate}
    [label=(\arabic*)]
    \item\label{item:clean1} every $N^W$-path in~$W$ is $\gamma_i$-zero for all~${i \in Z}$, and
    \item\label{item:clean2} $W$ has no ${(\ell,\ell)}$-subwall which is $\gamma_i$-bipartite for all~${i \in [m] \setminus Z}$. 
\end{enumerate}

\begin{lemma}
    \label{lem:cleansubwall}
    Let~${\Gamma = \prod_{i \in [m]} \Gamma_i}$ be a product of~$m$ abelian groups, 
    let~${(G,\gamma)}$ be a $\Gamma$-labelled graph, 
    let~${\psi \colon \{0\} \cup [m+1] \to \mathbb{N}_{\geq 3}}$ be a function, 
    and let~$W$ be a wall of order $\psi(0)+2$ in~$G$. 
    Then there exist a $\Gamma$-labelling~$\gamma'$ of~$G$ shifting-equivalent to~$\gamma$, 
    a subset~$Z$ of~${[m]}$, 
    and a ${(\gamma',Z,\psi(\abs{Z}+1)+2)}$-clean $\branch(W)$-anchored ${(\psi(\abs{Z}),\psi(\abs{Z}))}$-subwall of~$W$.
\end{lemma}

\begin{proof}
    Let~$Z$ be a maximal subset of~${[m]}$ such that there is a $(\psi(\abs{Z})+2,\psi(\abs{Z})+2)$-subwall~$W'$ of~$W$ which is $\gamma_i$-bipartite for all~${i \in Z}$. 
    Such a set~$Z$ exists because~${Z := \emptyset}$ satisfies the requirement.
    Since~$Z$ is maximal, 
    there is no~${j \in [m] \setminus Z}$ such that~$W'$ has a ${(\psi(\abs{Z}+1)+2,\psi(\abs{Z}+1)+2)}$-subwall which is $\gamma_j$-bipartite.

    Among all $\Gamma$-labellings~$\gamma'$ of~$G$ shifting-equivalent to~$\gamma$, 
    we choose~$\gamma'$ maximising the number of elements~${i \in Z}$ such that all corridors of~$W'$ are $\gamma'_i$-zero.
    If there is~${i \in Z}$ such that some corridor of~$W'$ is not $\gamma'_i$-zero, then 
    Lemma~\ref{lem:shifting} applied to~$\Gamma_i$ yields the $\Gamma$-labelling~$\gamma''$ for which every corridor of~$W'$ is~$\gamma''_i$-zero, 
    thus contradicting the choice of~$\gamma'$.
    Thus all corridors of~$W'$ are $\gamma_i'$-zero for all~${i \in Z}$.
    
    By Remark~\ref{rmk:nicesubwall}, $W'$ has a $\branch(W')$-anchored ${(\psi(\abs{Z}),\psi(\abs{Z}))}$-subwall~$W''$. 
    Then~$W''$ is $\branch(W)$-anchored since~${\branch(W') \subseteq \branch(W)}$. 
    Now the property~\ref{item:clean1} holds since every~$N^{W''}$-path in~$W''$ is a corridor of~$W'$. 
\end{proof}

In a sense, the notion of cleanness helps us to generalise the ideas of Thomassen~\cite{Thomassen1988} who proved the following result. 

\begin{proposition}[Thomassen~\cite{Thomassen1988}]
    \label{prop:thomassen}
    There exists a function~${w_{\ref*{prop:thomassen}} \colon \mathbb{N}^2 \to \mathbb{N}}$ satisfying the following. 
    Let~$t$ and~${w \geq 3}$ be integers, let~$\Gamma$ be an abelian group generated by an element of order at most~$t$, 
    and let~${(W, \gamma)}$ be a $\Gamma$-labelled wall of order~${w_{\ref*{prop:thomassen}}(t,w)}$. 
    Then~$W$ contains a ${(w,w)}$-subwall~$W'$ such that~${\gamma(P) = 0}$ for all corridors~$P$ of~$W'$. 
\end{proposition}
    
We extend Proposition~\ref{prop:thomassen} to a group generated by a fixed number of generators.

\begin{lemma}
    \label{lem:smallorder}
    There exists a function~$w_{\ref*{lem:smallorder}} \colon \mathbb{N}^3 \to \mathbb{N}$ satisfying the following. 
    Let~$q$, $t$ and~${w \geq 3}$ be integers, let~$\Gamma$ be an abelian group generated by~$q$ elements 
    each of order at most~$t$, 
    and let~${(W, \gamma)}$ be a $\Gamma$-labelled wall of order $w_{\ref*{lem:smallorder}}(q, t, w)$. 
    Then~$W$ contains a ${(w,w)}$-subwall~$W'$ such that~${\gamma(P) = 0}$ for all corridors~$P$ of~$W'$.
\end{lemma}

\begin{proof}
    We define 
    \begin{itemize}
        \item ${w_{\ref*{lem:smallorder}}(1,t,w) := w_{\ref*{prop:thomassen}}(t,w)}$, and 
        \item ${w_{\ref*{lem:smallorder}}(q,t,w) := w_{\ref*{prop:thomassen}}(t,f_{\ref*{lem:smallorder}}(q-1,t,w))}$ for all integers~${q \geq 2}$.
    \end{itemize}
    We prove the lemma by induction on~$q$ with Proposition~\ref{prop:thomassen} as the base case. 
    So let~${q \geq 2}$, and
    let~${\Gamma = \gen{ \{ x_i \colon i \in [q] \}}}$ for a suitable set of~$q$ generators each of order at most~$t$. 
    Let~${\Gamma_1 := \gen{x_q}}$ and~${\Gamma_2 := \gen{\{x_i \colon i \in [q-1]\}}}$, and let~$\gamma'$ be a $\Gamma_1$-labelling of~$W$ such that for every edge~$e$ of~$W$, we have~${\gamma(e) + \gamma'(e) \in \Gamma_2}$. 
    By Proposition~\ref{prop:thomassen}, $W$ has a  ${(w_{\ref*{lem:smallorder}}(q-1,t,w),w_{\ref*{lem:smallorder}}(q-1,t,w))}$-subwall~$W'$ such that~${\gamma(P) \in \Gamma_2}$ for all corridors~$P$ of~$W'$.     
    By the induction hypothesis, $W'$ has a ${(w,w)}$-subwall~$W''$ such that~${\gamma(P) = 0}$ for all corridors~$P$ of~$W''$. 
    Note that as~$W'$ is a subwall of~$W$, all corridors of~$W''$ in~$W'$ are corridors of~$W''$ in~$W$. 
\end{proof}

The following variation allows us to take advantage of our notion of cleanness and will be needed for Lemma~\ref{lem:omega-avoiding-cycle}. 

\begin{corollary}
    \label{cor:smallorder}
    There exists a function~${w_{\ref*{cor:smallorder}} \colon \mathbb{N}^3 \to \mathbb{N}}$ satisfying the following. 
    Let~$q$, $t$ and~${w \geq 3}$ be integers, let~$\Gamma$ be an abelian group, and let~$\Lambda$ be a subgroup of~$\Gamma$ generated by~$q$ elements each of order at most~$t$, 
    and let~${(W, \gamma)}$ be a $\Gamma$-labelled wall of order $w_{\ref*{cor:smallorder}}(q, t, w)$.
    If~${\gamma(O) \in \Lambda}$ for all cycles~$O$ of~$W$, 
    then~$W$ contains a $\gamma$-bipartite ${(w,w)}$-subwall. 
\end{corollary}

\begin{proof}
    Let~${w_{\ref*{cor:smallorder}}(q,t,w) := w_{\ref*{lem:smallorder}}(t^q,2t^q,w)}$.
    For each~$g$ in~${\Lambda \cap 2 \Gamma}$, let $\sigma(g)$ be an element of~$\Gamma$ such that~${2 \sigma(g) = g}$. 
    Let~$\hat{\Lambda}$ be the subgroup of~$\Gamma$ generated by~${S := (\Lambda \setminus 2\Gamma) \cup \{ \sigma(g) \colon g \in \Lambda \cap 2\Gamma \}}$. 
    Note that~${\Lambda \subseteq \hat{\Lambda}}$, 
    that~${\abs{S} \leq \abs{\Lambda} \leq t^q}$, 
    and that each element in~$S$ has order at most~${2 \abs{\Lambda} \leq 2t^q}$. 
    We will show that there is a $\Gamma$-labelling~$\gamma'$ 
    shifting-equivalent to~$\gamma$ such that~${\gamma'(P) \in \hat{\Lambda}}$ for every corridor~$P$ of~$W$.
    
    The wall~$W$ is a subdivision of some $3$-connected planar graph~${\hat{H}}$. 
    Let~$T$ be a spanning tree of~${\hat{H}}$, rooted at an arbitrary vertex~$r$. 
    Choose a $\Gamma$-labelling~$\gamma'$ shifting-equivalent to~$\gamma$ and a subtree~$T'$ of~$T$ containing~$r$ such that~${\gamma'(P) \in \hat{\Lambda}}$ for all corridors~$P$ of~$W$ corresponding to edges in~$T'$, and subject to these conditions,~$\abs{V(T')}$ is maximised.
    
    Suppose that~${T' \neq T}$. 
    Then there is an edge~$vw$ of~$T$ such that~${v \in V(T')}$ and~${w \notin V(T')}$. 
    Let~$Q$ be the corridor of~$W$ corresponding to the edge~$vw$. 
    Since~$\hat{H}$ is $3$-connected, there is a cycle~$O$ in~${W-E(Q)}$ containing~$v$ and~$w$. 
    Let~$O_1$ and~$O_2$ denote the distinct cycles in~${O \cup Q}$ containing~$Q$. 
    Since~$\gamma'$ is shifting-equivalent to~$\gamma$, 
    from the assumption on $W$ we deduce that ${\gamma'(O), \gamma'(O_1), \gamma'(O_2) \in \Lambda}$.
    Hence, ${2 \gamma'(Q) = \gamma'(O_1) + \gamma'(O_2) - \gamma'(O)\in \Lambda}$ and so $\sigma(2\gamma'(Q))$ is well defined.
    Observe that 
    \[
        2\left(\sigma(2\gamma'(Q))-\gamma'(Q)\right) 
        =  2\gamma'(Q) -2\gamma'(Q) 
        = 0.
    \]
    Let~$\gamma''$ be the $\Gamma$-labelling of~$G$ obtained from~$\gamma'$ by shifting by~${\sigma(2\gamma'(Q))-\gamma'(Q)}$ at~$w$. 
    Then $\gamma''(Q)=\gamma'(Q)+\sigma(2\gamma'(Q))-\gamma'(Q)=\sigma(2\gamma'(Q))\in \hat\Lambda$.
    Let~${T'' = T[V(T')\cup \{w\}]}$. 
    Then~${\gamma''(P) \in \hat{\Lambda}}$ for all corridors~$P$ of~$W$ corresponding to edges of~$T''$, contradicting our choice of~$\gamma'$ and~$T'$. Therefore $T'=T$.
    
    Now, observe that~${\gamma'(P) \in \hat{\Lambda}}$ for every corridor~$P$ of~$W$, 
    because~${\Lambda \subseteq \hat{\Lambda}}$ and for every cycle~$O$ of~$W$, we have~${\gamma'(O) = \gamma(O) \in \Lambda}$. 
    Let~$\gamma''$ be the $\hat{\Lambda}$-labelling of~$W$ which assigns an arbitrary edge~$e_P$ of each corridor~$P$ the value~${\gamma'(P)}$, and all other edges the value~$0$. 
    By Lemma~\ref{lem:smallorder}, there is a ${(w,w)}$-subwall~$W'$ of~$W$ which in particular is $\gamma''$-bipartite, and hence $\gamma'$-bipartite. 
    Since~$\gamma'$ and~$\gamma$ are shifting-equivalent, $W'$ is $\gamma$-bipartite.
\end{proof}

\section{Handling handles}
\label{sec:handles}

This section is dedicated to proving the following key lemma, which allows us to iteratively find sets of vertex-disjoint handles. 

\begin{lemma}
    \label{lem:addlinkage}
    There exist functions~${w_{\ref*{lem:addlinkage}}\colon \mathbb{N}^2 \to \mathbb{N}}$ and~$f_{\ref*{lem:addlinkage}} \colon \mathbb{N} \to \mathbb{N}$ satisfying the following. 
    Let~${k, t}$ and~$c$ be positive integers with~${c \geq 3}$, let~${\Gamma}$ be an abelian group, and let~${(G, \gamma)}$ be a $\Gamma$-labelled graph.
    Let~$W$ be a wall in~$G$ of order at least~${w_{\ref*{lem:addlinkage}}(k,c)}$ such that all corridors of~$W$ are $\gamma$-zero. 
    For each~${i \in [t-1]}$, let~$\mathcal{P}_i$ be a set of~$4k$ $W$-handles in~$G$ such that the paths in~${\bigcup_{i \in [t-1]} \mathcal{P}_i}$
    are vertex-disjoint. 
    If~$G$ contains at least~${f_{\ref*{lem:addlinkage}}(k)}$ vertex-disjoint $\gamma$-non-zero ${\branch(W)}$-paths, 
    then there exist a $c$-column-slice~$W'$ of~$W$ and a set~$\mathcal{Q}_i$ of~$k$ vertex-disjoint $W'$-handles for each~${i \in [t]}$ such that 
    \begin{enumerate}
        [label=(\roman*)]
        \item for each~${i \in [t-1]}$, the set~$\mathcal{Q}_i$ is a subset of the row-extension of~$\mathcal{P}_i$ to~$W'$ in~$W$, 
        \item the paths in~${\bigcup_{i \in [t]} \mathcal{Q}_i}$ are vertex-disjoint,
        \item the paths in~$\mathcal{Q}_t$ are $\gamma$-non-zero. 
    \end{enumerate}
\end{lemma}

Before we can prove this lemma, we need to establish a variety of other lemmas. 
At the heart of the proof, we have the following natural result, which we will iteratively apply to decouple the sets of vertex-disjoint handles that we will construct. 
Huynh, Joos, and Wollan~\cite[Lemma~27]{HuynhJW2017} proved a somewhat similar result for oriented group-labelled graphs.

\begin{lemma}
    \label{lem:separating}
    Let~${k, t}$ be positive integers, 
    let~${\Gamma}$ be an abelian group, 
    let~${(G,\gamma)}$ be a $\Gamma$-labelled graph, 
    and let~$T$ be a subset of~${V(G)}$. 
    For each ${i \in [t-1]}$, let~$\mathcal{P}_i$ be a set of $T$-paths of size~${4k}$ 
    such that the paths in~${\bigcup_{i \in [t-1]} \mathcal{P}_i}$ are vertex-disjoint. 
    
    If~$G$ contains~$k$ vertex-disjoint ${\gamma}$-non-zero $T$-paths, 
    then there exist a set~$\mathcal{Q}_t$ of $k$ vertex-disjoint ${\gamma}$-non-zero $T$-paths and a subset~${\mathcal{Q}_i \subseteq \mathcal{P}_i}$ of size~$k$ for each~${i \in [t-1]}$ 
    so that the paths in~${\bigcup_{i \in [t]} \mathcal{Q}_i}$ are vertex-disjoint. 
\end{lemma}

\begin{proof}
    Let~$\mathcal{Q}_t$ be a set of~$k$ vertex-disjoint $\gamma$-non-zero $T$-paths such that 
    the number of edges of paths in~${\mathcal{Q}_t}$ that are not contained in any path in~${\mathcal{P} := \bigcup_{i \in [t-1]} \mathcal{P}_i}$ is as small as possible. 
    For each~${j \in [t-1]}$, let~$\mathcal{P}_j^\ast$ be the set of paths in~$\mathcal{P}_j$ that do not contain an endvertex of a path in~$\mathcal{Q}_t$. 
    We have~${\abs{\mathcal{P}_j^\ast} \geq \abs{\mathcal{P}_j}-2\abs{\mathcal{Q}_t} = 2k}$.
    Let~$\mathcal{P}_j^{**}$ be the set of all paths in~$\mathcal{P}_j^\ast$
    intersecting a path in~$\mathcal{Q}_t$. 
    
    Assume that for some~${j \in [t-1]}$ we have~${\abs{ \mathcal{P}_j^{**} } \geq k+1}$. 
    Then there are two paths~${P_1, P_2 \in \cP_j^{**}}$ such that when traversing from an endvertex~$p_i$ of~$P_i$ for each~${i \in [2]}$, $P_1$ and~$P_2$ first meet the same path~${Q \in \mathcal{Q}_t}$. 
    For each~${i \in [2]}$, let~$q_i$ be the first intersection of~$P_i$ and~$Q$ when traversing~$P_i$ from~$p_i$. 
    
    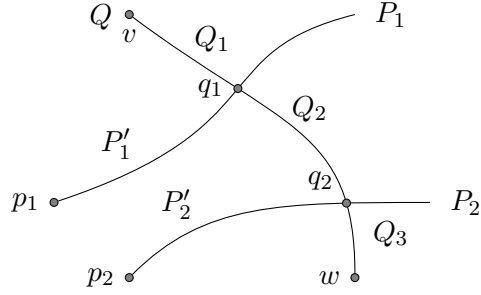
\begin{figure}
        \begin{tikzpicture}
            \tikzstyle{w}=[circle,draw,fill=black!50,inner sep=0pt,minimum width=3pt]
            \node at (0,2) [w,label=left:$p_1$] (vp1) {};
            \node at (1,1) [w,label=left:$p_2$] (vp2) {};
            \node at (1,4.5) [w,label=below:$v$] (v){};
            \node at (4,1) [w,label=left:$w$] (w){};
            \draw [name path=p1] (vp1) .. controls (3,3) and (2,4) ..  (4,4.5)
            node[pos=0.1,label=$P_1'$]{}
            node [pos=1,label=right:$P_1$]{};
            \draw [name path=q] (v) .. controls (3,3) and (4,3) .. (w)
                node [pos=0,label=left:$Q$]{}
                node [pos=0.55,label=$Q_2$]{} 
                node [pos=0.2,label=$Q_1$]{} 
                node [pos=0.9,label=right:$Q_3$]{} ;
            \draw [name path=p2] (vp2) .. controls (2,2) and (3,2) .. (5,2)
                node [pos=0.2,label=$P_2'$]{}
                node [pos=1,label=right:$P_2$]{};
            \node [name intersections={of=p1 and q}] at (intersection-1) [w,label=left:$q_1$](vq1) {};
            \node [name intersections={of=p2 and q}] at (intersection-1) [w,label=above left:$q_2$](vq2) {};
        \end{tikzpicture}
        \caption{Segments of the paths~$P_1$,~$P_2$ and~$Q$ mentioned in Lemma~\ref{lem:separating}.}
        \label{fig:separating}
    \end{figure}
    
    Let~$v$ and~$w$ be the endvertices of~$Q$ such that the distance between~$v$ and~$q_1$ in~$Q$ is smaller than the distance between~$v$ and~$q_2$ in~$Q$. 
    Let~${Q_1, Q_2, Q_3}$ be the subpaths of~$Q$ from~$v$ to~$q_1$, from~$q_1$ to~$q_2$, and 
    from~$q_2$ to~$w$, respectively. 
    Also, for each~${i \in [2]}$, let~$P_i'$ be the subpath of~$P_i$ from~$p_i$ to~$q_i$,
    see Figure~\ref{fig:separating}.
    Since the paths of~${\bigcup_{i \in [t-1]} \mathcal{P}_i}$ are vertex-disjoint and~${v, w \notin V(P_1 \cup P_2)}$, both~$Q_1$ and~$Q_3$ contain edges not in a path of~${\bigcup_{i \in [t-1]} \mathcal{P}_i}$; for instance, edges incident with~$q_1$ or~$q_2$.

    By assumption, ${\gamma(Q) = \gamma(Q_1) + \gamma(Q_2) + \gamma(Q_3)}$ is non-zero. 
    If there is a~${\{v, w, p_1, p_2\}}$-path~$R$ in~${Q \cup P_1' \cup P_2'}$ such that~${R \neq Q}$ and~${\gamma(R) \neq 0}$, 
    then by replacing~$Q$ with~$R$, the number of edges of paths in~$\mathcal{Q}_t$ that are not contained in any path in~${\bigcup_{i \in [t-1]} \mathcal{P}_i}$ decreases. 
    Therefore, by the assumption on~$\mathcal{Q}_t$, we have~${\gamma(R) = 0}$ for every such path~$R$. 
    It implies that
    \begin{enumerate}
        [label=(\arabic*)]
        \item \label{item:separating1} ${\gamma(Q_1)+\gamma(Q_2)+\gamma(P_2') = 0}$, 
        \item \label{item:separating2} ${\gamma(P_1')+\gamma(Q_2)+\gamma(P_2') = 0}$, 
        \item \label{item:separating3} ${\gamma(P_1')+\gamma(Q_2)+\gamma(Q_3) = 0}$.
    \end{enumerate}
    The equations~\ref{item:separating1} and~\ref{item:separating2} imply that~${\gamma(Q_1) = \gamma(P_1')}$ 
    and, similarly, the equations~\ref{item:separating2} and~\ref{item:separating3} imply that~${\gamma(Q_3) = \gamma(P_2')}$.
    But these imply that 
    \[
        0 = \gamma(P_1')+\gamma(Q_2)+\gamma(P_2') = \gamma(Q_1) + \gamma(Q_2) + \gamma(Q_3) \neq 0, 
    \]
    which is a contradiction. 
    We conclude for all~${j \in [t-1]}$ that~${\abs{ \mathcal{P}_j^{**} } \leq k}$, 
    and thus~${\abs{\mathcal{P}_j^\ast\setminus \mathcal{P}_j^{**}} \geq k}$.
    
    For each~${j \in [t-1]}$, let~$\mathcal{Q}_j$ be a set of~$k$ paths in~${\mathcal{P}_j^\ast\setminus \mathcal{P}_j^{**}}$. 
    Then for each~${i \in [t-1]}$, we have that~${\mathcal{Q}_i \subseteq \mathcal{P}_i}$, and
    the paths in~${\bigcup_{i \in [t]} \mathcal{Q}_i}$
    are vertex-disjoint, as required. 
\end{proof}

Lemma~\ref{lem:addlinkage} mentions a set of vertex-disjoint $\branch(W)$-paths in~$G$, but note that these may arbitrarily intersect the internal vertices of corridors of~$W$. 
The following technical lemma allows us to take subpaths of these paths which intersect the corridors of~$W$ in a more controlled manner. 

\begin{lemma}
    \label{lem:breaking}
    Let~${\Gamma}$ be an abelian group, 
    let~${(G, \gamma)}$ be a $\Gamma$-labelled graph 
    and let~${H \subseteq G}$ be a subdivision of a $3$-connected graph 
    such that every corridor of~$H$ is $\gamma$-zero. 
    If~$G$ contains a $\gamma$-non-zero $\branch(H)$-path~$P$, then 
    there exist a subpath~$U$ of~$P$ 
    and a set~$\mathcal{X}$ of at most~$12$ corridors of~$H$ satisfying the following properties:
    \begin{enumerate}[label=(\roman*)]
        \item\label{item:b1} ${H \cap U}$ is a subgraph of~${\bigcup \mathcal{X}}$.
        \item\label{item:b2} For any subgraph~${H' \subseteq H}$ which is a subdivision of a $3$-connected graph with~${\bigcup \mathcal{X} \subseteq H'}$, 
        and any subset~${T \subseteq \branch(H')}$ 
        with~${\abs{T} \geq 3}$, 
        there is a $\gamma$-non-zero $T$-path in~${H'\cup U}$.  
    \end{enumerate} 
\end{lemma}

\begin{proof}
    For each vertex~$z$ of~$H$, we define~${x_{z,1}, x_{z,2} \in \Gamma}$ 
    and a path~$X_z$ as follows.
    \begin{itemize}
        \item If~$z$ has degree~$2$ in~$H$, then 
        let~$X_z$ be the corridor of~$H$ containing~$z$, 
        let~${x_{z,1} := \gamma(X_{z,1})}$,
        and~${x_{z,2} := \gamma(X_{z,2})}$ 
        where~$X_{z,1}$ and~$X_{z,2}$ are the two distinct subpaths of~$X_z$ from~$z$ to the endvertices of~$X_z$.
        \item Otherwise let~$X_z$ be a path of length~$0$ containing~$z$, let~${x_{z,1} := 0}$, and~${x_{z,2} := 0}$.
    \end{itemize}

    A path~$Q$ from~${a \in V(H)}$ to~${b \in V(H)}$ in~$G$  
    with~${x_{a,1} = x_{a,2}}$ and~${x_{b,1} = x_{b,2}}$ 
    is \emph{$\gamma$-preserving} if~${x_{a,1} + \gamma(Q) + x_{b,1} = 0}$, 
    and is \emph{$\gamma$-breaking} otherwise. 
    We first prove the following claim. 
    
    \begin{claim*}
        ${P \cup H}$ contains a $\gamma$-breaking path~$U$ such that 
        \begin{enumerate}
            [label=(\alph*)]
            \item\label{item:breaking1} 
                both endvertices of~$U$ are in~$H$, 
            \item\label{item:breaking3} 
                at most two corridors of~$H$ intersect the set of internal vertices of~$U$, and
            \item\label{item:breaking4} 
                for each endvertex~$z$ of~$U$, either~${z \in \branch(H)}$
                or~$X_z$ contains no internal vertex of~$U$.
        \end{enumerate}
    \end{claim*}

    \begin{proofofclaim}
        Suppose that this claim does not hold. 
        We first show that 
        \begin{enumerate}[label=($\ast$)]
            \item \label{item:breakingast} if~$P$ contains a $V(H)$-path~$Q$ from~$a$ to~$b$ where~${X_a \neq X_b}$, then~$Q$ is a $\gamma$-preserving path.
        \end{enumerate}
        Because there are no two distinct corridors of~$H$ with the same set of endvertices, $X_a$ intersects at most one of~$X_{b,1}$ and~$X_{b,2}$.
        If~${\gamma(X_{a,1}) \neq \gamma(X_{a,2})}$ and $X_{b,i}$ does not intersect~$X_a$ for some~${i \in \{1,2\}}$, then ${X_{a,1} \cup Q \cup X_{b,i}}$ or~${X_{a,2} \cup Q \cup X_{b,i}}$ is $\gamma$-breaking.
        It is not difficult to verify that such a $\gamma$-breaking path satisfies the required properties, contradicting the assumption.
        Thus, ${\gamma(X_{a,1}) = \gamma(X_{a,2})}$, and by symmetry, 
        ${\gamma(X_{b,1})= \gamma(X_{b,2})}$. 
        Now, by the assumption, $Q$ is $\gamma$-preserving.
        This shows~\ref{item:breakingast}. 
        \smallskip
        
        Let~$M_1$ be the set of $V(H)$-paths in~$P$ whose endvertices are internal vertices of distinct corridors of~$H$. 
        Let~$M_2$ be the set of maximal subpaths of~${P - \bigcup_{Q \in M_1} E(Q)}$ of length at least~$1$.
        Note that for each~${R \in M_2}$,
        at most one corridor of~$H$ intersects the set of internal vertices of~$R$.
        
        Let~$v$ and~$w$ be the endvertices of~$P$, and let~${t := \abs{M_1}+\abs{M_2}}$. 
        Note that $M_1\cup M_2$ is a partition of~$P$ into $t$ edge-disjoint subpaths, each having length at least~$1$.
        Let~$P_1$ be the path in~${M_1 \cup M_2}$ containing~$v$, and for each~${i \in [t-1]}$ let~$P_{i+1}$ be the unique path in~${(M_1 \cup M_2) \setminus \{P_j \colon j \in [i]\}}$ sharing an endvertex, say~$v_i$, with~$P_i$. 
        By~\ref{item:breakingast}, we have~${\gamma(X_{v_i,1}) = \gamma(X_{v_i,2})}$ for all~${i \in [t-1]}$. 
        Note that~${\gamma(X_{v,2}) = \gamma(X_{w,1}) = 0}$. 
        
        We claim that there is a $\gamma$-breaking path in~$M_2$.
        Suppose for a contradiction that all paths in~$M_2$ are $\gamma$-preserving. 
        By~\ref{item:breakingast}, all paths in~$M_1$ are $\gamma$-preserving and therefore 
        \[
            \sum_{i=0}^{t-1} \big( \gamma(X_{v_i,2}) + \gamma(P_{i+1}) + \gamma(X_{v_{i+1},1}) \big) = 0. 
        \]
        As every corridor of~$H$ is $\gamma$-zero,
        we know that 
        \[
            \sum_{i=0}^{t-1} \big( \gamma(X_{v_i,2}) + \gamma(X_{v_{i+1},1}) \big) = 2 \sum_{i =1}^{t-1} \gamma(X_{v_i}) = 0. 
        \]
        This implies that~${\sum_{i =0}^{t-1} \gamma(P_{i+1}) = \gamma(P) = 0}$, which contradicts the fact that~$P$ is $\gamma$-non-zero.
        So, we conclude that there exists~${j \in [t]}$ such that~${P_j \in M_2}$ and~$P_j$ is $\gamma$-breaking.
        
        We obtain that the path~$P'$ defined by
        \[
            P' :=
            \begin{cases}
                P_1 \cup P_2 & \textnormal{ if } j = 1,\\
                P_{j-1} \cup P_j\cup P_{j+1} & \textnormal{ if } j \in [t-1] \setminus \{1\},\\
                P_{t-1}\cup P_t & \textnormal{ if } j = t,
            \end{cases}
        \]
        has the desired properties. 
    \end{proofofclaim}
    
    Let~$U$ be a path obtained by the previous claim.
    Let~$\mathcal{X}_1$ be the set of corridors of~$H$ intersecting the set of internal vertices of~$U$. 
    By the previous claim, we have~${\abs{\mathcal{X}_1} \leq 2}$.

    Let~$a$,~$b$ be the endvertices of~$U$.
    For~${x \in \{a,b\}}$, let~$Y_x$ be a corridor of~$H$ containing~$x$. 
    Thus if $x$ has degree~$2$ in~$H$, then~${Y_x = X_x}$
    and otherwise~$Y_x$ is an arbitrary corridor of~$H$ ending at~$x$.
    Let~$\mathcal{X}_2$ be a minimal set of corridors of~$H$ such that~${Y_a, Y_b \in \mathcal{X}_2}$ and each endvertex of~$Y_a$ and~$Y_b$ is contained in at least three corridors in~$\mathcal{X}_2$.
    Then~${\abs{\mathcal{X}_2} \leq 10}$. 

    Let~${\mathcal{X} := \mathcal{X}_1 \cup \mathcal{X}_2}$. 
    Then~${\abs{\mathcal{X}} \leq 12}$ and~\ref{item:b1} holds. 
    It remains to show~\ref{item:b2}.
    Let~$H'$ be a subgraph of~$H$ which is a subdivision of a $3$-connected graph~$\hat{H}'$
    such that~${\bigcup \mathcal{X}}$ is a subgraph of~$H'$
    and let~$T$ be subset of~${\branch(H')}$ of size at least~$3$. 
    Note that~${\branch(H') \subseteq \branch(H)}$ and 
    therefore every corridor of~$H'$ is $\gamma$-zero.
    By the construction of~$\mathcal{X}_2$,
    each endvertex of~$U$ is contained in some corridor of~$H$ which is also a corridor of~$H'$ and every corridor of~$H$ intersecting the set of internal vertices of~$U$ is also a corridor of~$H'$.
    Hence from the claim, we deduce that 
    \begin{enumerate}
        [label=(\alph*$'$)]
        \item\label{item:breaking1'} 
            both endvertices of~$U$ are in~$H'$, 
        \item\label{item:breaking3'} 
            at most two corridors of~$H'$ intersect the set of internal vertices of~$U$, and
        \item\label{item:breaking4'} 
            for each endvertex~$z$ of~$U$, either~${z \in \branch(H')}$
            or the corridor of $H'$ containing $z$ contains no internal vertex of~$U$.
    \end{enumerate}
    
    Since~$\hat{H}'$ is $3$-connected, there are two disjoint paths~$Q_1$, $Q_2$ in~$H'$ between the endvertices of~$U$ and the set~$T$. 
    If~${Q_1 \cup Q_2}$ does not contain an internal vertex of~$U$, then~${Q_1 \cup U \cup Q_2}$ is as desired, since~$U$ is $\gamma$-breaking and all corridors of~$H$ are $\gamma$-zero.

    If~${Q_1 \cup Q_2}$ contains an internal vertex of~$U$, then~${Q_1 \cup Q_2}$ contains a corridor~$R$ of~$H'$ intersecting the set of internal vertices of~$U$.
    Choose~$x$ among two endvertices of~$R$ that is closer to~$T$ in~${Q_1 \cup Q_2}$. 
    Then~$x$ is not an endvertex of~$U$.
    Since~$x$ is in at least~$3$ corridors of~$H'$, by property~\ref{item:breaking3'},~${x \notin V(U)}$.
    
    Since~$\hat{H}'$ is $3$-connected, by properties~\ref{item:breaking3'} and~\ref{item:breaking4'}, ${H'-E(\bigcup \mathcal{X}_1)}$ is connected.
    Thus,~$H'$ has a path~$Q$ connecting the endvertices of~$U$ which contains no internal vertex of~$U$. 
    The cycle~${O := Q \cup U}$ is $\gamma$-non-zero since~$U$ is $\gamma$-breaking. 
    Since~${\hat{H}'}$ is~$3$-connected, there are three vertex-disjoint paths from~$T$ to~${\{x\} \cup \branch(U)}$ in~$H'$.
    By extending one of the paths ending at~$x$ to an internal vertex of~$U$ through~$R$ if~${x \notin V(O)}$, 
    we obtain three vertex-disjoint~${(V(O),T)}$-paths in~${H' \cup U}$. 
    Hence Lemma~\ref{lem:non-zero-path-in-cycle} yields the desired result.
\end{proof}

In the next lemma, we extend subpaths from the previous lemma to handles of some suitable column-slice. 

\begin{lemma}
    \label{lem:addlinkage1}
    There exist functions~${w_{\ref*{lem:addlinkage1}}\colon \mathbb{N}^2 \to \mathbb{N}}$ and~$f_{\ref*{lem:addlinkage1}} \colon \mathbb{N} \to \mathbb{N}$ satisfying the following. 
    Let~$k$ and~$c$ be positive integers with~${c \geq 3}$, let~${\Gamma}$ be an abelian group, and let~${(G, \gamma)}$ be a $\Gamma$-labelled graph.
    Let~$W$ be a wall in~$G$ of order at least~${w_{\ref*{lem:addlinkage1}}(k,c)}$ such that all corridors of~$W$ are $\gamma$-zero. 
    If~$G$ contains ${f_{\ref*{lem:addlinkage1}}(k)}$ vertex-disjoint
    $\gamma$-non-zero
    $\branch(W)$-paths, 
    then there exist a $c$-column-slice~$W'$ of~$W$ and~$k$ vertex-disjoint $\gamma$-non-zero $W'$-handles in~$G$. 
\end{lemma}

\begin{proof}
    Let~${h(k) := 3 f_{\ref*{thm:tpath}}(k) + 1}$
    and~${f_{\ref*{lem:addlinkage1}}(k) := 2 \cdot 240^2 h(k)^2 }$. 
    Let~$\mathcal{P}$ be a set of~${f_{\ref*{lem:addlinkage1}}(k)}$ vertex-disjoint $\gamma$-non-zero $\branch(W)$-paths. 
    Let~${w_{\ref*{lem:addlinkage1}}(k,c) := (48 h(k) + 1)(c - 1) + 144h(k) + 1 }$.

    \begin{claim*}
        For all~${i \in [h(k)]}$ there exist a set~$\mathcal{C}_i$ of $3$-column-slices of~$W$, a set~$\mathcal{R}_i$ of $3$-row-slices of~$W$ and a subpath~$U_i$ of a path in~$\mathcal{P}$ such that, with ${H_i := \bigcup(\mathcal{C}_i \cup \mathcal{R}_i)}$, we have 
        \begin{enumerate}
            [label=(\alph*)]
            \item\label{item:slicenumbers} ${1 \leq \abs{\mathcal{C}_i} \leq 48}$ and~${1 \leq \abs{\mathcal{R}_i} \leq 48}$,
            \item\label{item:disjointcolumns} every~${C \in \mathcal{C}_i}$ is vertex-disjoint from every~${C' \in \mathcal{C}_j}$ for all~${j \in [i-1]}$,
            \item\label{item:disjointrows} every~${R \in \mathcal{R}_i}$ is vertex-disjoint from every~${R' \in \mathcal{R}_j}$ for all~${j \in [i-1]}$, 
            \item\label{item:vertexdisjU} $U_i$ and~$U_j$ are vertex-disjoint for all~${j \in [i-1]}$, 
            \item\label{item:pathintersections-col} every $3$-column-slice of~$W$ that intersects~$U_i$ also intersects some $3$-column-slice in~${\bigcup_{j \in[i]} \mathcal{C}_{j}}$, 
            \item\label{item:gammabreakingTpath} for any column-slice~$W'$ of~$W$ which is disjoint from~$U_i$, there is a $\gamma$-non-zero $W'$-handle in~${H_i \cup U_i}$.
        \end{enumerate}
    \end{claim*}
    
    \begin{proofofclaim}[Proof of Claim]
        We proceed by induction on~$i$. 
        For~${i \in [h(k)]}$, assume that the claim holds for all~${j \in [i-1]}$. 
        We define
        \begin{itemize}
            \item $\widetilde{\mathcal{C}}_i$ to be the set of all $3$-column-slices of~$W$ which intersect no $3$-column-slices in~${\bigcup_{j \in [i-1]} \mathcal{C}_{j}}$,
            \item $\widetilde{\mathcal{R}}_i$ to be the set of all $3$-row-slices of~$W$ which intersect no $3$-row-slices in~${\bigcup_{j \in [i-1]} \mathcal{R}_{j}}$, and
            \item ${\widetilde{H}_i := \bigcup \left( \widetilde{\mathcal{C}}_i \cup \widetilde{\mathcal{R}}_i \right)}$.
        \end{itemize}

        We will first show that the number of vertices in~${\branch(W) \setminus V(\widetilde{H}_i)}$ is small. 
        Let~$I$ be the set of all column indices~$a$ such that 
        the $a$-th column~$C^W_a$ intersects no $3$-column-slice in~${\bigcup_{j \in [i-1]} \mathcal{C}_j}$. 
        Then~$I$ admits a partition into intervals consisting of consecutive integers such that the number of intervals is bounded by~${\sum_{j \in [i-1]} \abs{\mathcal{C}_j} + 1 \leq 48(i-1)+1}$. 
        Observe that  at least one interval of~$I$ has size at least~$3$ because~${w_{\ref*{lem:addlinkage1}}(k,c) > 3 \cdot 48 (i-1) + 2 \cdot (48 (i-1) + 1)}$.
        This implies that~$\widetilde{C}_i$ is nonempty. 
        Suppose that a vertex~$v$ in~${\branch(W)}$ is not in~${\widetilde{H}_i}$. 
        Let us say that~${v \in V(C^W_x) \cap V(R^W_y)}$ for some~$x$ and~$y$. 
        Since~$v$ is not in~$\widetilde{H}_i$, either~$C^W_x$ intersects some $3$-column-slice in~$\mathcal{C}_j$ for some~${j < i}$ 
        or~$x$ belongs to an interval of~$I$ of size at most~$2$. 
        Since at least one interval of~$I$ has size more than~$2$, 
        the number of possible values of~$x$ is at most 
        \[
            3 \cdot \sum_{j \in [i-1]} \abs{\mathcal{C}_j} + 
            2 \cdot \sum_{j \in [i-1]} \abs{\mathcal{C}_j}
            \leq 240(i-1).
        \]
        
        By the same argument, 
        we deduce that~$\widetilde{R}_i$ is nonempty
        and the number of possible values of~$y$ is at most~${240(i-1)}$. 
        Thus, the number of vertices in~${\branch(W)}$ not in~$\widetilde{H}_i$ is at most~${2(240(i-1))^2}$, 
        because there are at most two vertices of~${\branch(W)}$ in~${V(C^W_x) \cap V(R^W_y)}$ for each~$x$ and~$y$.
    
        Since~${\abs{\mathcal{P}} \geq 2(240i)^2 > 2(240(i-1))^2+(i-1)}$, 
        there is a path~$P_i$ in~$\mathcal{P}$ both of whose endvertices are in~$\widetilde{H}_i$ such that~${U_j}$ is not a subpath of~${P_i}$ for all~${j < i}$.
        Let~$\mathcal{X}_i$ be the set of at most~$12$ corridors of~$\widetilde{H}_i$ and let~$U_i$ be a subpath of~${P_i}$ guaranteed by Lemma~\ref{lem:breaking}. 

        Let~$\mathcal{C}_i$ be a minimal non-empty subset of~$\widetilde{C}_i$ containing all $3$-column-slices in~$\widetilde{\mathcal{C}}_i$
        which intersect some corridor in~$\mathcal{X}_i$. 
        Since each corridor of~$\widetilde{H}_i$ intersects at most four $3$-column-slices,~${\abs{\mathcal{C}_i} \leq 4 \cdot 12}$. 
        Similarly, let~$\mathcal{R}_i$ be a minimal non-empty subset of $\widetilde{R}_i$ containing all $3$-row-slices in~$\widetilde{\mathcal{R}}_i$ which intersect some corridor in~$\mathcal{X}_i$. 
        Then~${\abs{\mathcal{R}_i} \leq 48}$. 
        Now~\ref{item:slicenumbers}--\ref{item:pathintersections-col} are true by construction. 
        
        To see~\ref{item:gammabreakingTpath},
        let~$W'$ be a column-slice disjoint from~$U_i$. 
        Note that~${H_i \cup W'}$ is a subdivision of a $3$-connected graph. 
        Applying Lemma~\ref{lem:breaking}\ref{item:b2} 
        with~${H' := H_i \cup W'}$ and~${T := \branch(H') \cap V(W')}$, 
        there is a $\gamma$-non-zero $T$-path~$P$ in~${H' \cup U_i}$. 
        Then~$P$ must use at least one edge of~$U_i$ because ${\gamma(P') = 0}$ for every $T$-path~$P'$ in~$H'$. 
        This implies that ${P \subseteq H_i \cup U_i}$, because~$T$ separates~$U_i$ from~$W'$ in~$H'$. 
        It follows that~$P$ is a $W'$-handle.
    \end{proofofclaim}

    Let~$I$ be the set of all column indices~$a$ such that 
    $C^W_a$ intersects no $3$-column-slice in~${\bigcup_{j \in [h(k)]} \mathcal{C}_j}$. 
    By~\ref{item:slicenumbers},~$I$ admits a partition into at most~${48h(k)+1}$ disjoint intervals, each consisting of consecutive integers. 
    Since~${w_{\ref*{lem:addlinkage1}}(k,c) - 3 \cdot 48 h(k) > (48 h(k) + 1)(c-1)}$, there exist~$c$ consecutive columns of~$W$ that do not intersect any $3$-column-slice in~${\bigcup_{i \in [h(k)]} \mathcal{C}_i}$. 
    Thus they form a $c$-column-slice~$W'$ of~$W$ which is disjoint from~${\bigcup_{i \in [h(k)]} U_i}$ by~\ref{item:pathintersections-col}.
	By~\ref{item:gammabreakingTpath}, for each~${i \in[h(k)]}$, there is a $\gamma$-non-zero $W'$-handle~$X_i$ in~${H_i \cup U_i}$.
	
    By construction, each vertex of~${W - V(W')}$ is contained in at most two graphs in~${\{ H_i \colon i \in [h(k)] \}}$ 
    and in at most one path in~${\{ U_i \colon i \in [h(k)] \}}$. 
    Hence, every vertex of~$W$ is contained in at most three paths in~${\{ X_i \colon i \in [h(k)] \}}$.
    Since~${h(k) = 3 f_{\ref*{thm:tpath}}(k) + 1}$, any vertex set of size at most~${f_{\ref*{thm:tpath}}(k)}$ cannot hit all these paths. 
    By Theorem~\ref{thm:tpath}, there exist~$k$ vertex-disjoint $\gamma$-non-zero $W'$-handles, as desired. 
\end{proof}

Now we obtain Lemma~\ref{lem:addlinkage} as a corollary of Lemmas~\ref{lem:separating} and~\ref{lem:addlinkage1}. 

\begin{proof}[Proof of Lemma~\ref{lem:addlinkage}]
    Let~$w_{\ref*{lem:addlinkage}}$ and~$f_{\ref*{lem:addlinkage}}$ be the functions~$w_{\ref*{lem:addlinkage1}}$ and~$f_{\ref*{lem:addlinkage1}}$ respectively, as given by Lemma~\ref{lem:addlinkage1}. 
    By Lemma~\ref{lem:addlinkage1}, 
    there exists a $c$-column-slice~$W'$ of~$W$ and a set~$\mathcal{P}'$ of~$k$ vertex-disjoint $\gamma$-non-zero $W'$-handles. 
    For each~${i \in [t-1]}$ and each~${P \in \mathcal{P}_i}$, let~${R_{P,1}}$ and~${R_{P,2}}$ be the rows of~$W$ containing the endvertices of~$P$, 
    and let~$Q_P$ be the row-extension of~$P$ to $W'$, that is the unique $V(W')$-path such that~${P \subseteq Q_P \subseteq P \cup R_{P,1} \cup R_{P,2}}$. 
    We remark that it is possible that~${R_{P,1} = R_{P,2}}$.
    
    Now applying Lemma~\ref{lem:separating} to the sets~${\{ Q_P \colon P \in \mathcal{P}_i \}}$ for~${i \in [t-1]}$ together with~$\mathcal{P}'$ yields the desired result.
\end{proof}

\section{Basic lemmas for products of abelian groups}
\label{sec:abelian}

In this section, we prove some basic lemmas on products of abelian groups that will be useful throughout Sections~\ref{sec:handles2cycles} and~\ref{sec:main}.

\medskip

An \emph{arithmetic progression} is a set of integers~$A$ such that there are integers~$a$ and~${b \neq 0}$ for which~${A = \{ a + bn \colon n \in \mathbb{Z} \}}$. 
For a set~${\mathcal{A} = \{A_i \colon i \in [k]\}}$ of arithmetic progressions, we say~$\mathcal{A}$ \emph{covers} a set~$S$ if~${S \subseteq \bigcup_{i \in[k]} A_i}$.
We will use the following fact about arithmetic progressions, conjectured by Erd\H{o}s in 1962 and proven by Crittenden and Vanden Eynden~\cite{CrittendenVE1970} in 1969.
We cite an equivalent version of Balister et~al.~\cite{BalisterBMST2020}, who presented a simple proof. 

\begin{theorem}[\cite{CrittendenVE1970,BalisterBMST2020}]
    \label{thm:arithmetic}
    Let~${\mathcal{A} = \{ A_i \colon i \in [k] \}}$ be a set of~$k$ arithmetic progressions. 
    If~$\mathcal{A}$ covers a set of~$2^k$ consecutive integers, then~$\mathcal{A}$ covers~$\mathbb{Z}$. 
\end{theorem}

\begin{corollary}
    \label{cor:omega-avoiding}
    Let~$m$, $t$, and~$\omega$ be positive integers, let~${\Gamma = \prod_{j \in [m]} \Gamma_j}$ be a product of~$m$ abelian groups, and for all~${j \in [m]}$ let~$\Omega_j$ be a subset of~$\Gamma_j$ of size at most~$\omega$. 
    For all~${i \in [t]}$ and~${j \in [m]}$, let~${g_{i,j}}$ be an element of~$\Gamma_j$.
    If for each~${i \in [t]}$ there exists an integer~$c_i$ 
    such that~${\sum_{i=1}^t c_i g_{i,j} \notin \Omega_j}$ 
    for all~${j \in [m]}$, 
    then for each~${i \in [t]}$ there exists an integer~${d_i \in [2^{m\omega}]}$ 
    such that~${\sum_{i=1}^t d_i g_{i,j} \notin \Omega_j}$ for all~${j \in [m]}$. 
\end{corollary}

\begin{proof}
    Pick an integer~$d_i$ for each~${i \in [t]}$ such that~${\sum_{i=1}^t d_i g_{i,j} \notin \Omega_j}$ for all~${j \in [m]}$, and subject to this ${\abs{\{i \in [t] \colon d_i \in [2^{m\omega}]\}}}$ is maximised. 
    Suppose for a contradiction that for some~${x \in [t]}$, we have that ${d_x \notin [2^{m\omega}]}$. 
    Without loss of generality, we may assume~${x = t}$.
    
    For all~${j \in [m]}$ and~${g \in \Omega_j}$, let~$A_{j,g}$ be the set of integers~$d$ such that~${d g_{t,j} + \sum_{i=1}^{t-1} d_i g_{i,j} = g}$. 
    Note that~$A_{j,g}$ is an arithmetic progression or contains at most one integer. 
    Let $A_{j,g}'$ be an arithmetic progression such that~${A_{j,g} \subseteq A_{j,g}'}$ and~${d_t \notin A_{j,g}'}$. 
    Such an~$A_{j,g}'$ exists; 
    if~$A_{j,g}$ is an arithmetic progression then let~${A'_{j,g} := A_{j,g}}$, and if~$A_{j,g}$ contains a unique integer~$a_j$, 
    then let~${A'_{j,g}}$ be the arithmetic progression~${\{ a_j + 2 (d_t - a_j) k \colon k \in \mathbb{Z}\}}$.

    Now~${\mathcal{A} := \{ A'_{j,g} \colon j \in [m], g \in \Omega_j \}}$ is a set of~${m\omega}$ arithmetic progressions not covering~$d_t$. 
    By Theorem~\ref{thm:arithmetic}, there exists~${d'_t \in [2^{m\omega}]}$ such that~${d'_t g_{t,j} + \sum_{i=1}^{t-1} d_i g_{i,j} \notin \Omega_j}$ for all~${j \in [m]}$, contradicting our choice of~$d_t$.
\end{proof}

\medskip

For a sequence~${\mathfrak{a} = ( a_i \colon i \in [t] )}$ over an abelian group~$\Gamma$, 
we let~${\Sigma(\mathfrak{a})}$ denote the set of all sums of subsequences of~$\mathfrak{a}$. 
We write~${\abs{\mathfrak{a}} := t}$, the length of~$\mathfrak{a}$. 
We say~${a \in \Gamma}$ is \emph{repeated} in~$\mathfrak{a}$ if~${a = a_i = a_j}$ for some~${1 \leq i < j \leq t}$. 
We say that~$\mathfrak{a}$ is \emph{good} if 
$\abs{\Sigma(\mathfrak{a})} \geq \abs{\mathfrak{a}}$. 
Obviously, a sequence of pairwise distinct elements of~$\Gamma$ is good. 
Also, observe that for an element~$g$ of order at least~${t}$, if~${a_i := g}$ for all~${i \in [t]}$, then the sequence~${\mathfrak{a} = (a_i \colon i \in [t])}$ is good, 
because ${\abs{\Sigma(\mathfrak{a})} \geq \abs{ \{kg \colon k \in [t]\}} \geq t}$.

\begin{lemma}\label{lem:smallgoodset}
    Let~$\Gamma$ be an abelian group and let~${\mathfrak{a} = (a_i \colon i \in [t])}$  be a sequence of length~$t$ over~$\Gamma$.
    If all repeated elements in~$\mathfrak{a}$ have order at least~$t$, then~$\mathfrak{a}$ is good. 
\end{lemma}

\begin{proof}
    We proceed by induction on~$t$. 
    We may assume that~${t \geq 2}$.
    If~$\mathfrak{a}$ has no repeated elements, then~${\{ a_i \colon i \in [t]\} \subseteq \Sigma(\mathfrak{a})}$ and therefore~$\mathfrak{a}$ is good. 
    Thus, 
    without loss of generality, we may assume that~$a_t$ is a repeated element. 
    Let~${\mathfrak{a}' := (a_i \colon i \in [t-1])}$
    and~${S := \Sigma(\mathfrak{a}')}$.   
    By induction~${\abs{S} \geq t-1}$.
    We may assume that~${\abs{S} = t-1}$.
    Let~${T := \{ x + a_t \colon x \in S \} \subseteq \Sigma(\mathfrak{a})}$. 
    If~${S = T}$, then~${\sum_{x \in S} x = \sum_{x \in S} (x+a_t)}$ 
    and therefore~${\abs{S}a_t = 0}$, contradicting the assumption on the order of~$a_t$.
    Thus~${S \neq T}$ and therefore~${\abs{\Sigma(\mathfrak{a})} \geq \abs{S \cup T} \geq t}$. 
\end{proof}

\medskip

The following lemma and its corollary are useful to find a 
cycle whose $\gamma_i$-value is not in~$\Omega_i$ for all~${i \in [m]}$. 

\begin{lemma}
    \label{lem:vectorsum}
    Let~$m$,~$t$, and~$\omega$ be positive integers, let~${\Gamma = \prod_{j \in [m]} \Gamma_j}$ be a product of~$m$ abelian groups and for all~${j \in [m]}$ let~$\Omega_j$ be a subset of~$\Gamma_j$ of size at most~$\omega$.
    For all~${i \in [t]}$ let~$S_i$ be a subset of~$\Gamma$ 
    such that for each~${j \in [m]}$ there exists some~${i \in [t]}$ such that~${\pi_j(g) \neq \pi_j(g')}$ 
    for all distinct~${g, g'}$ in~$S_i$.
    If~${\abs{S_i} > m\omega }$ for all~${i \in [t]}$,
    then for every~${h \in \Gamma}$ there is a sequence~${(g_i \colon i \in [t])}$ of elements of~$\Gamma$ such that 
    \begin{enumerate}
        [label=(\roman*)]
        \item ${g_i \in S_i}$ for each ${i \in [t]}$, and 
        \item ${\pi_j\left(h + \sum_{i \in [t]} g_i \right) \notin \Omega_j}$ for all~${j \in [m]}$.
    \end{enumerate}
\end{lemma}

\begin{proof}
    Uniformly at random, select~${g_i \in S_i}$ independently for each~${i \in [t]}$, and consider the sum ${g := h + \sum_{i \in [t]} g_i}$. 
    For each~${j \in [m]}$, there exists~${i \in [t]}$ such that~$g$ and every group element obtained by replacing~$g_i$ in the sum with a different element of~$S_i$ have distinct $j$-th coordinates. 
    Hence, the probability that~${\pi_j(g) \in \Omega_j}$ is at most~${\omega/(m\omega+1)}$. 
    It follows that there is a positive probability that~${\pi_j(g) \notin \Omega_j}$ for all~${j \in [m]}$.
\end{proof}

For two sequences $\mathfrak{a} = (a_i \colon i \in[t])$, $\mathfrak{b} = (b_i \colon i \in [t])$ of length~$t$ over a product~${\Gamma = \prod_{j \in [m]} \Gamma_j}$ of~$m$ abelian groups,
we write~${\mathfrak{a} - \mathfrak{b}}$ to denote the sequence~${(a_i - b_i \colon i \in [t])}$
and for~${j \in [m]}$, we write~${\pi_j(\mathfrak{a})}$ to denote the sequence~${(\pi_j(a_i) \colon i \in [t])}$ over~$\Gamma_j$.

\begin{corollary}
    \label{cor:vectorsum}
    Let~$m$ and~$\omega$ be positive integers, 
    let~${\Gamma = \prod_{i \in [m]} \Gamma_i}$ be a product of~$m$ abelian groups, 
    and for each~${i \in [m]}$ let~$\Omega_i$ be a subset of~$\Gamma_i$ of size at most~$\omega$ and let~${\mathfrak{a}_i := (a_{i,j} \colon j \in [m\omega+1])}$ and ${\mathfrak{b}_i := (b_{i,j} \colon j \in [m\omega+1])}$ be two sequences over~${\Gamma}$
    such that~${ \pi_i(\mathfrak{a}_{i}-\mathfrak{b}_{i})}$ is good. 
    Then for all~${h \in \Gamma}$, ${i \in [m]}$, and~${j \in [m\omega+1]}$, 
    there exists~${c_{i,j} \in \{a_{i,j},b_{i,j}\}}$ 
    such that for all~${x \in [m]}$, we have~${\pi_x\left(h + \sum_{i \in [m]} \sum_{j \in [m\omega+1]} c_{i,j}\right)\notin \Omega_x}$. 
\end{corollary}

\begin{proof}
    By definition of a good sequence, for each~${i \in [m]}$, we have
    ${\abs{\Sigma(\pi_i(\mathfrak{a}_{i} - \mathfrak{b}_{i}))} \geq m\omega+1}$. 
    Let~$S_i$ be a subset of~${\Sigma(\mathfrak{a}_{i} - \mathfrak{b}_{i})}$ 
    such that~$\pi_i$ restricted to~$S_i$ is a bijection from~$S_i$ to~${\Sigma(\pi_i(\mathfrak{a}_{i} - \mathfrak{b}_{i}))}$. 
    Let~${h' := h+\sum_{i \in [m]} \sum_{j \in [m\omega+1]} b_{i,j}}$.
    We apply Lemma~\ref{lem:vectorsum} with~$h'$ to find a sequence~${(g_i \colon i \in [m])}$ such that~${g_i \in S_i}$ for each~${i \in [m]}$ 
    and~${\pi_j\left(h' + \sum_{i \in [m]} g_i\right)\notin \Omega_j}$ for all $j\in [m]$. 
    Now, for each~${i \in [m]}$, we have that ${g_i \in \Sigma(\mathfrak{a}_i - \mathfrak{b}_i)}$, and hence for each~${j \in [m\omega+1]}$ there exists~${c_{i,j} \in \{a_{i,j},b_{i,j}\}}$ such that~${g_i+\sum_{j \in [m\omega+1]} b_{i,j} = \sum_{j \in [m\omega+1]} c_{i,j}}$.
    This completes the proof.
\end{proof}

\medskip

Given positive integers~$n$ and~$k$, we write~${R(n;k)}$ for the minimum integer~$N$ such that in every $k$-colouring of the edges of~$K_N$ there is a monochromatic copy of~$K_n$. 
A classical result of Ramsey~\cite{Ramsey1929} shows that~${R(n;k)}$ exists. 

\begin{lemma}
    \label{lem:ramsey}
    There exists a function~$f_{\ref*{lem:ramsey}} \colon \mathbb{N}^2 \to \mathbb{N}$ satisfying the following. 
    Let~$m$,~$t$ and~$N$ be positive integers 
    with~${N \geq f_{\ref*{lem:ramsey}}(t,m)}$ 
    and let~${\Gamma = \prod_{i \in [m]} \Gamma_i}$ be a product of~$m$ abelian groups. 
    Then for every sequence~${(g_i \colon i \in [N])}$ over~$\Gamma$, 
    there exists a subset~$I$ of~$[N]$ with~${\abs{I} = t}$ such that for each~${i \in [m]}$, either
	\begin{itemize}
		\item ${\pi_i(g_{j}) = \pi_i(g_{k})}$ for all ${j,k \in I}$, or 
		\item ${\pi_i(g_{j}) \neq \pi_i(g_{k})}$ for all distinct~${j,k \in I}$.
	\end{itemize}
	Furthermore, if~$Z$ is a subset of~${[m]}$ such that for all distinct~$i$ and~$j$ in~${[N]}$ there exists ${x \in Z}$ such that~${\pi_x(g_i) \neq \pi_x(g_j)}$, then the second condition holds for some~${i \in Z}$. 
\end{lemma}

\begin{proof}
    Let~${f_{\ref*{lem:ramsey}}(t,m) := R(t; 2^m)}$. 
    We define a $2^m$-colouring of the edges of~$K_N$ by colouring each edge~${xy}$ of~$K_N$ by the set~${\{ i \in [m] \colon \pi_i(g_x) \neq \pi_i(g_y)\}}$. 
    The result follows from the definition of~${R(t;2^m)}$. 
    Note that if~$Z$ is subset of~${[m]}$ such that for all distinct~$i$ and~$j$ in~${[N]}$ there exists an~${x \in Z}$ such that~${\pi_x(g_i) \neq \pi_x(g_j)}$, then every set used in the colouring intersects~$Z$. 
\end{proof}

\section{From handles to cycles}
\label{sec:handles2cycles}

The focus of this section is proving the following key lemma, which will be the final ingredient needed in the proof of Theorem~\ref{thm:main} for constructing 
the cycles from the clean subwall from Section~\ref{sec:cleanwalls} and the sets of handles from Section~\ref{sec:handles}. 

\begin{lemma}
    \label{lem:omega-avoiding-cycle}
    There exist functions ${c_{\ref*{lem:omega-avoiding-cycle}}, r_{\ref*{lem:omega-avoiding-cycle}} \colon \mathbb{N}^4 \to \mathbb{N}}$ satisfying the following. 
    Let~$t$,~$\ell$,~$m$ and~$\omega$ be positive integers with~${\ell \geq 3}$, 
    let~${\Gamma = \prod_{i \in [m]} \Gamma_i}$ be a product of~$m$ abelian groups, for each~${i \in [m]}$ let~$\Omega_i$ be a subset of~$\Gamma_i$ of size at most~$\omega$, 
    and let~${(G,\gamma)}$ be a $\Gamma$-labelled graph. 
    Let~$Z$ be a subset of~${[m]}$, 
    let~$W$ be a ${(\gamma,Z,\ell)}$-clean $(c,r)$-wall 
    with~${c \geq c_{\ref*{lem:omega-avoiding-cycle}}(t, \ell, m, \omega)}$ 
    and~${r \geq r_{\ref*{lem:omega-avoiding-cycle}}(t, \ell, m, \omega)}$. 
    For every set~$\mathcal{P}$ of at most~$t$ vertex-disjoint $W$-handles such that~${\gamma_i\left( \bigcup \mathcal{P} \right) \notin \Omega_i}$ for all~${i \in Z}$,
    there is a cycle~$O$ in~${W \cup \bigcup \mathcal{P}}$ such that~${\gamma_i(O) \notin \Omega_i}$ for all~${i \in [m]}$. 
\end{lemma}

We begin by linking up a set of handles of a wall into a cycle. 

\begin{lemma}
    \label{lem:pairlink}
    Let~$t$ be a positive integer and 
    let~$W$ be a ${(c,r)}$-wall in a graph~$G$ with~${r \geq 3}$ and~${c \geq \max\{3,t+1\}}$. 
    For every set~$\mathcal{P}$ of at most~$t$ vertex-disjoint $W$-handles in~$G$, 
    there is a cycle~$O$ in~${W \cup \bigcup \mathcal{P}}$ that contains~${\bigcup \mathcal{P}}$ as a subgraph. 
\end{lemma}

\begin{proof}
    Let~$T$ be the set of endvertices of all paths in~$\mathcal{P}$. 
    We proceed by induction on~$t$. 
    If~${\abs{\mathcal{P}} \leq 2}$, then 
    as~${W \cup \bigcup \mathcal{P}}$ is $2$-connected, ${W \cup \bigcup \mathcal{P}}$ has a cycle~$O$ containing at least one edge from every path in~$\mathcal{P}$. 
    Therefore, we may assume that~${\abs{\mathcal{P}} = t > 2}$. 
    
    By symmetry, we may assume that the first column of~$W$ meets at least two paths in~$\cP$. 
    In the first column of~$W$, choose two degree-$2$ nails~$v_1$,~$v_2$ that are endvertices of distinct paths~$P_1$,~$P_2$ of~$\cP$ respectively
    such that the distance between~$v_1$ and~$v_2$ in the first column of~$W$ is minimised.
    Let~$Q$ be the path from~$v_1$ to~$v_2$ in the first column of~$W$.
    Let~$P^\ast$ be the path~${P_1 \cup Q \cup P_2}$. 
    Let~$W'$ be the ${(c-1)}$-column-slice of~$W$ obtained by removing the first column. 
    Let~$\mathcal{P}'$ be the row-extension of~${(\cP \setminus \{P_1, P_2\}) \cup \{ P^\ast \}}$ to~$W'$. 
    Since~$\mathcal{P}'$ is a set of~${t-1}$ vertex-disjoint $W'$-handles, 
    by the induction hypothesis, there is a cycle~$O$ in~${W' \cup \bigcup \mathcal{P}'}$ such that~${ \bigcup \mathcal{P}' \subseteq O}$. 
    This completes the proof because~${\bigcup \mathcal{P} \subseteq \bigcup \mathcal{P}'}$ and~${W' \cup \bigcup \mathcal{P}' \subseteq W \cup \bigcup \mathcal{P}}$.
\end{proof}

In order to obtain a 
cycle whose $\gamma_i$-value is not in~$\Omega_i$ for every~${i \in [m]}$, 
it will be useful to have access to a sequence of subwalls to reroute segments of the cycle constructed by the previous lemma. 
The following straightforward corollary provides this. 

\begin{corollary}
    \label{cor:cycle+subwalls}
    There exist functions~${c_{\ref*{cor:cycle+subwalls}}, r_{\ref*{cor:cycle+subwalls}} \colon \mathbb{N}^3 \rightarrow \mathbb{N}}$ satisfying the following.
    Let~$t$,~$k$ and~$w$ be positive integers with~${w \geq 4}$.
    Let~$G$ be a graph containing a $(c,r)$-wall~$W$ with~${c \geq c_{\ref*{cor:cycle+subwalls}}(t,k,w)}$ and~${r \geq r_{\ref*{cor:cycle+subwalls}}(t,k,w)}$. 
    For every set~$\mathcal{P}$ of at most~$t$ vertex-disjoint $W$-handles in~$G$,
    there exist a cycle~$O$ in~${W \cup \bigcup \mathcal{P}}$ and a set~${\{ W_i \colon i \in [k] \}}$ of~$k$ vertex-disjoint $N^W$-anchored ${(w,w)}$-subwalls of~$W$ such that~${\bigcup \mathcal{P} \subseteq O}$ 
    and~${W_i \cap O = R_1^{W_i}}$ for all~${i \in [k]}$. 
\end{corollary}

\begin{proof}
    Define
    \[
        {c_{\ref*{cor:cycle+subwalls}}(t,k,w) :=  kw+t+1, } \ 
        \textnormal{ and } \ 
        {r_{\ref*{cor:cycle+subwalls}}(t,k,w) := (2t+1)(w-2) + 1}.
    \]
    The case~${\mathcal{P} = \emptyset}$ is easy to verify, so we may assume~${\abs{\mathcal{P}} = t > 0}$. 
    Without loss of generality, we may assume that the last column of~$W$ intersects~$\bigcup \mathcal{P}$. 
    Let~$W''$ be a ${kw}$-column-slice of~$W$ containing the last column of~$W$
    and let~$W'$ be a ${(c-kw)}$-column-slice of~$W$ disjoint with~$W''$.
    Let~$\mathcal{P}'$ denote the row-extension of~$\mathcal{P}$ to~$W'$ in~$W$. 
    
    By the pigeonhole principle, there is a ${(w-1)}$-row-slice of~$W''$ which is disjoint from~${\bigcup \mathcal{P}'}$. 
    Hence there is a $w$-row-slice~$S$ of~$W''$ such that~${S \cap \bigcup \mathcal{P}' = R_1^{S}}$. 
    We can pack~$k$ vertex-disjoint $N^W$-anchored ${(w,w)}$-subwalls~${\{ W_i \colon i \in [k]\}}$ in~$S$ so that~${W_i \cap \bigcup \mathcal{P}' = R_1^{W_i}}$. 
    Applying Lemma~\ref{lem:pairlink} to~$W'$ and~$\mathcal{P}'$ yields the desired cycle. 
\end{proof}

Moreover, we need the following variation of Lemma~\ref{lem:non-zero-path-in-cycle}. 

\begin{lemma}
    \label{lem:reroutingpath}
    Let~$\Gamma$ be an abelian group, and let~${(G,\gamma)}$ be a $\Gamma$-labelled graph,
    let~$O$ be a $\gamma$-non-zero cycle in~$G$, 
    and let~$P$ be a path disjoint from~$O$.
    If~$G$ contains three vertex-disjoint ${(V(P),V(O))}$-paths~$P_1$, $P_2$, $P_3$, 
    then there is a path~$P'$ in~${P \cup O \cup P_1 \cup P_2 \cup P_3}$ with the same endvertices as~$P$ such that~${\gamma(P') \neq \gamma(P)}$.
\end{lemma}

\begin{proof}
    Let~${T := V(O) \cap (V(P_1)\cup V(P_2)\cup V(P_3))}$.
    Since~${\abs{T} = 3}$, the cycle~$O$ contains three distinct $T$-paths~$Q_1$, $Q_2$, $Q_3$ such that~${V(Q_i) \cap V(P_i) = \emptyset}$ for each~${i \in [3]}$. 
    Since~${\gamma(O) \neq 0}$, we have that~${\gamma(O) \neq 2\gamma(O)}$, which implies 
    \[
        \gamma(Q_1) + \gamma(Q_2) + \gamma(Q_3) \neq (\gamma(Q_2) + \gamma(Q_3)) + (\gamma(Q_3) + \gamma(Q_1)) + (\gamma(Q_1) + \gamma(Q_2)).
    \]
    Without loss of generality, we may assume~${\gamma(Q_3) \neq \gamma(Q_1) + \gamma(Q_2)}$. 
    Observe that there are paths~$P'$ and~$P''$ in~${P \cup O \cup P_1\cup P_2}$ with the same endvertices as~$P$ such that~${E(P') \setminus E(P'') = E(Q_3)}$ and~${E(P'') \setminus E(P') = E(Q_1) \cup E(Q_2)}$. 
    Hence, $P'$ or~$P''$ is the desired path.
\end{proof}

Finally, we can prove Lemma~\ref{lem:omega-avoiding-cycle}. 

\begin{proof}[Proof of Lemma~\ref{lem:omega-avoiding-cycle}]
    For convenience, let~${w_0 := w_{\ref*{cor:smallorder}}(m\omega,m\omega,\ell)}$. 
    We define 
    \begin{align*}
        c_{\ref*{lem:omega-avoiding-cycle}}(t,\ell,m,\omega) &:= 
        \max\{ 3, t+1, c_{\ref*{cor:cycle+subwalls}}(t, \, m(m\omega+1), \, w_0 + 1)\},\\
        r_{\ref*{lem:omega-avoiding-cycle}}(t,\ell,m,\omega) &:= 
        \max\{3, r_{\ref*{cor:cycle+subwalls}}(t, \, m(m\omega+1), \, w_0 + 1)\}.
    \end{align*}
    If~${Z = [m]}$, then the result follows from Lemma~\ref{lem:pairlink}. 
    Hence, without loss of generality, we may assume that~${Z = [m] \setminus [y]}$ for some~${y \in [m]}$. 
    By Corollary~\ref{cor:cycle+subwalls}, 
    there exist a cycle~$O$ in~${W \cup \bigcup \mathcal{P}}$ 
    and a set~${\{ W_{i,j} \colon i\in [y],\, j \in  [y\omega+1] \}}$ of~$y(y\omega+1)$ vertex-disjoint $N^W$-anchored $(w_0+1,w_0+1)$-subwalls of~$W$  
    such that~${\bigcup \mathcal{P} \subseteq O}$ 
    and~${W_{i,j} \cap O = R_1^{W_{i,j}} =: P_{i,j}}$ for 
    all~${i \in [y]}$ and~${j \in [y\omega+1]}$. 
    For~${i \in [y]}$ and~${j \in [y\omega+1]}$, 
    let~$W'_{i,j}$ be a ${w_0}$-row-slice of~$W_{i,j}$ disjoint from~$O$,
    and let~$H$ be the graph obtained from~$O$ by deleting the internal vertices of the paths~$P_{i,j}$ for 
    all~${i \in [y]}$ and~${j \in [y\omega+1]}$. 
    
    For each~${i \in [y]}$, we now recursively define a family~${\big(Q_{i,j} \colon j \in [y\omega+1]\big)}$ of paths 
    and a family  
    ${(S_{i,j} \colon j \in [y\omega+1])}$ of subsets of~$\Gamma_i$, 
    such that for all~${j \in [y\omega+1]}$ and~${g \in S_{i,j}}$,
    \begin{itemize}
        \item ${\abs{S_{i,j}} \leq j-1}$, and
        \item the order of~$g$ is at most~${m \omega}$.
    \end{itemize}  
    We first set~${S_{i,1} := \emptyset}$. 
    Now, for~${j \in [y\omega+1]}$, let~$\lambda_j$ be the induced~${\Gamma_i/\gen{S_{i,j}}}$-labelling of~$G$. 
    Note that since~${i \notin Z}$ and~$W$ is ${(\gamma,Z,\ell)}$-clean, $W$ has no  $\gamma_i$-bipartite ${(\ell,\ell)}$-subwall, and in particular, each~${W_{i,j}'}$ does not have such a subwall.
    As ${\abs{S_{i,j}} \leq m\omega}$ and each element of~$S_{i,j}$ has order at most~${m \omega}$, 
    Corollary~\ref{cor:smallorder} implies that there is a $\lambda_j$-non-zero cycle~$O_{i,j}$ in~$W'_{i,j}$. 
    By Lemma~\ref{lem:reroutingpath}, 
    there is a path~$Q_{i,j}$ in~${W_{i,j}}$ with the same endvertices as~$P_{i,j}$ such that~${\lambda_j(P_{i,j}) \neq \lambda_j(Q_{i,j})}$, since there are three vertex-disjoint ${(V(P_{i,j}),V(O_{i,j}))}$-paths in~$W_{i,j}$. 
    We set 
    \[
        S_{i,j+1} :=
        \begin{cases}
            S_{i,j} & \textnormal{ if } {\gamma_i(Q_{i,j}) - \gamma_i(P_{i,j})} \textnormal{ has order at least } {m\omega+1}, \textnormal{ and } \\
            S_{i,j} \cup \big\{\gamma_i(Q_{i,j}) - \gamma_i(P_{i,j})\big\} & \textnormal{ otherwise.}
        \end{cases}
    \]
    
    For~${i \in [y]}$, let~${D_i := \big(\gamma(Q_{i,j}) - \gamma(P_{i,j}) \colon j \in [y\omega+1]\big)}$. 
    By construction of the sets~${S_{i,j}}$, the projection of~${D_i}$ to~${\Gamma_i}$ contains no repeated elements of order at most~${m \omega}$ in~${\Gamma_i}$, and thus
    this sequence is good by Lemma~\ref{lem:smallgoodset}. 
    Hence, by Corollary~\ref{cor:vectorsum} with~${h := \gamma(H)}$, there exists~${X_{i,j} \in \{ P_{i,j}, Q_{i,j} \}}$ for all~${i \in [y]}$ and~${j \in [y\omega+1]}$
    such that for the cycle~${O' := H \cup \bigcup\big\{ X_{i,j} \colon i \in [y],\, j \in [y\omega+1] \big\}}$, we have that~${\gamma_i(O') \notin \Omega_i}$ for all~${i \in [y]}$. 
    By construction, ${\gamma_j(O') = \gamma_j(O) \notin \Omega_j}$ for all~${j \in Z}$ because~$W$ is ${(\gamma,Z,\ell)}$-clean.
\end{proof}

\section{Proof of the main theorem}
\label{sec:main}

We now complete the proof of Theorem~\ref{thm:main}, which we are restating for the convenience of the reader. 

\mainthm*

\begin{proof}
    Let~${\Gamma := \prod_{i \in [m]} \Gamma_i}$ denote the product of the~$m$ given abelian groups, 
    and let~$\gamma \colon E(G) \to \Gamma$ denote the $\Gamma$-labelling for which~${\gamma_i(e) = \pi_i (\gamma(e))}$ for all~${e \in E(G)}$. 
    We proceed by induction on~$k$. 
    For~${k \leq 2}$, we may trivially set~${f_{m,\omega}(k) = 1}$. 
    Now suppose that~${k > 2}$ and that there is some integer~${f_{m,\omega}(k-1)}$ as per the theorem. 
    For every subgraph~$H$ of~$G$, let~${\nu(H)}$ denote the maximum size of a set of cycles~$O$ in~$H$ with~${\gamma_i(O) \notin \Omega_i}$ for all~${i \in [m]}$ such that no three cycles in the set share a common vertex. 
    Observe that~$\nu$ is a packing function for~$G$. 
    We will show that if~$\tau_\nu(G)$ is sufficiently large relative to~$k$, then~${\nu(G) \geq k}$. 
    This will complete the proof of the theorem.
    
    For non-negative integers~$p$ and~$z_0$ with~${z_0 \leq m}$, let~${\alpha(p,z_0)}$ and~${\rho(z_0)}$ be recursively defined as follows. 
    For every non-negative integer~$p$, we define
    \begin{align*}
        \alpha(p,0) &:= k2^{m\omega}+m\omega+1, &
        \rho(0) &:= m + f_{\ref*{lem:ramsey}}(\alpha(1,0),m), 
    \intertext{and for~${z_0 > 0}$ we recursively define}
        \alpha(p,z_0) &:= 4^{\max\{\rho(z_0-1)-p,0\}} \alpha(1,z_0-1),&
        \rho(z_0) &:= m + f_{\ref*{lem:ramsey}}(\alpha(1,z_0),m). 
    \end{align*}
    Let~${\hat{p} := \rho(m-1)}$. 
    Note that we may assume that~$f_{\ref*{lem:ramsey}}$ is increasing in its first argument, and hence~${\rho(z_0) \leq \hat{p}}$ for~${z_0 \leq m-1}$. 
    Let 
    \[
        u := \max\{ \lceil  f_{m,\omega}(k-1)/3\rceil, 
        f_{\ref*{thm:tpath}} ( f_{\ref*{lem:addlinkage}} ( f_{\ref*{lem:ramsey}}(\alpha(1,m),m) ) ) + 3 \}.
    \]
    
    We recursively define~${\beta(p,z_0,z)}$ for non-negative integers~$p$, $z_0$, and~$z$ with~${z_0 \leq z \leq m}$ and~${p \leq \hat{p}}$, 
    as well as~${\psi(z)}$ for a non-negative integer~$z$ with~${z \leq m+1}$ as follows. 
    We define 
    \begin{align*}
        \psi(m+1) &:= 3, 
        \intertext{and for~${z \leq m}$ we define} 
        \beta(p,z_0,z) &:=
        \begin{cases}
            \max \big\{ 
                u,\, 
                k c_{\ref*{lem:omega-avoiding-cycle}}(2^{m\omega}\hat{p}, \psi(z+1)+2, m, \omega) \big\}
            & \text{if } z_0 = 0,\\
            \beta(1,z_0-1,z) 
            & \text{if } z_0 > 0 \text{ and } p = \hat{p},\\
            \max\big\{
                \beta(p+1,z_0,z),\, w_{\ref*{lem:addlinkage}}(\alpha(p+1,z_0),
                \beta(p+1,z_0,z)) \big\} 
            & \text{if } z_0 > 0 \text{ and } p < \hat{p}; 
        \end{cases}\\
        \psi(z) &:= \max \big\{ 
            \psi(z+1),\, 
            \beta(0,z,z),\, 
            r_{\ref*{lem:omega-avoiding-cycle}}(2^{m\omega}\hat{p},\, 
            \psi(z+1)+2, m, \omega) \big\}.
    \end{align*}
    
    Observe that~${\beta(p,z_0,z) \geq u}$. 

    Lastly, we define~${f_{m,\omega}(k) := \max\big\{
        6f_{\ref*{thm:wall}}(\psi(0)+2),\,  
        6u, \,
        12f_{m,\omega}(k-1) 
        \big\}}$.
    
    \medskip
    
    Let~$T$ be a minimum $\nu$-hitting set of size~${t := \tau_\nu(G)}$, 
    and assume that~${t > f_{m,\omega}(k)}$. 
    By the induction hypothesis,~$G$ has a half-integral packing of~${k-1}$ cycles in~$\mathcal{O}$ 
    and therefore we may assume 
    for a contradiction 
    that~${\nu(G) = k-1}$. 
    For each subgraph~$H$ of~$G$, if~${\nu(H) < \nu(G)}$, then 
    by the induction hypothesis,~${\tau_\nu(H) \leq f_{m,\omega}(k-1) \leq f_{m,\omega}(k)/12 < t/12}$. 
    Lemma~\ref{lem:welllinked} yields that the set~$\mathcal{T}_T$ of all separations~${(A,B)}$ of~$G$ of order less than~${t/6}$ with~${\abs{B \cap T} > 5t/6}$ is a tangle of order~${\lceil t/6 \rceil > f_{\ref*{thm:wall}}(\psi(0)+2)}$. 
    By Theorem~\ref{thm:wall}, there is a ${(\psi(0)+2,\psi(0)+2)}$-wall in~$G$ dominated by~$\mathcal{T}_T$. 
    By Lemma~\ref{lem:cleansubwall}, 
    this wall has a  ${(\psi(\abs{Z}),\psi(\abs{Z}))}$-subwall~$W$ which is ${(\gamma',Z,\psi(\abs{Z}+1)+2)}$-clean
    for some subset~${Z \subseteq [m]}$ and some $\Gamma$-labelling~$\gamma'$ of~$G$  shifting-equivalent to~$\gamma$.
    By Lemma~\ref{lem:dominatedsubwall}, the wall~$W$ is 
    dominated by~$\mathcal{T}_T$. 
    Since~${\gamma(O) = \gamma'(O)}$ for every cycle~$O$ in~$G$, we may assume without loss of generality that~${\gamma = \gamma'}$. 
    
    \begin{claim}
        \label{clm:manyhandles}
        There exist integers~$c$ and~$p$ with~${c \geq \beta(1,0,z)}$ and~${0 \leq p \leq \hat{p}}$, 
        a $c$-column-slice~$W'$ of~$W$, 
        a family~${\mathfrak{P} = ( \mathcal{P}_i \colon i \in [p])}$ of non-empty sets of $W'$-handles, 
        and a family~${( Z_i \colon i \in \{0\} \cup [p] )}$ of disjoint subsets of~$Z$ such that 
        \begin{enumerate}
            [label=\textnormal{(\alph*)}, series=manyhandles]
            \item \label{item:manyhandles1} 
                if~${P \in \mathcal{P}_i}$ and~${Q \in \mathcal{P}_j}$ are not vertex-disjoint for some~${i,j \in [p]}$, then~${i = j}$ and~${P = Q}$, 
            \item \label{item:manyhandles2} 
                ${\bigcup_{i\in \{0\}\cup [p]}Z_i = Z}$, 
            \item \label{item:manyhandles2b}
                ${\abs{\mathcal{P}_i} \geq \alpha(p,\abs{Z_0})}$ for all~${i \in [p]}$, 
            \item \label{item:manyhandles3} 
                ${\abs{\gamma_j(\mathcal{P}_{i})} = \abs{\mathcal{P}_i}}$ for all~${i \in [p]}$ and~${j \in Z_i}$, 
            \item \label{item:manyhandles4} 
                ${\abs{\gamma_j(\mathcal{P}_{i})} = 1}$ for all~${i \in [p]}$ and ${j \in Z_{0}}$, 
            \item \label{item:manyhandles5} 
                there is some~${g \in \gen{\bigcup_{i \in [p]} \gamma(\mathcal{P}_{i})}}$ such that~${\pi_j(g) \notin \Omega_j}$ for all~${j \in Z_0}$. 
        \end{enumerate} 
    \end{claim}
    
    \begin{proofofclaim}
        For non-negative integers~$c$,~$q$, and~$p$, 
        we say that a triple ${(W',\mathfrak{P}, \mathcal{Z})}$ consisting of a wall~$W'$, a family $\mathfrak{P} := (\mathcal{P}_i \colon i \in [p])$ of non-empty sets of $W'$-handles, and a family $\mathcal{Z} := ( Z_i \colon i \in \{0\} \cup [p])$ of disjoint subsets of~$Z$ 
        is a \emph{${(c,q,p)}$-McGuffin} if~$W'$ is a $c$-column-slice of~$W$, 
        and~$W'$,~$\mathfrak{P}$, and~$\mathcal{Z}$ satisfy~\ref{item:manyhandles1}--\ref{item:manyhandles4}, as well as
        the following two conditions:
        \begin{enumerate}
            [resume*=manyhandles]
            \item \label{item:manyhandles7}
                ${Z_i\neq\emptyset}$ for~${i \in [q]}$ and~${Z_i=\emptyset}$ for~${i \in [p] \setminus[q]}$,  
            \item \label{item:manyhandles8}
                for all distinct~${i, i' \in [p] \setminus [q]}$ there is ${j \in Z_{0}}$ such that~${\gamma_j(\mathcal{P}_{i}) \cap \gamma_j(\mathcal{P}_{i'}) = \emptyset}$.
        \end{enumerate}
        Note that~${(W,\emptyset,(Z))}$ is a ${(\psi(\abs{Z}),0,0)}$-McGuffin. 
        Let~${(q,p)}$ be a lexicographically maximal pair of non-negative integers 
        with~${q \leq \abs{Z}}$ and~${q \leq p \leq \hat{p}}$ 
        for which there is a ${(c,q,p)}$-McGuffin~${(W',\mathfrak{P},\mathcal{Z})}$ for some~${c \geq \beta(p,\abs{Z_0},\abs{Z})}$. 
        Now suppose for a contradiction that~${(W',\mathfrak{P},\mathcal{Z})}$ does not satisfy~\ref{item:manyhandles5}. 
        Then~$Z_0$ is nonempty. 
        We distinguish two cases. 
        
        If~${p = \hat{p}}$, then ${p - q \geq \hat{p} - m \geq \rho(z_0) - m \geq  f_{\ref*{lem:ramsey}}(\alpha(q+1,\abs{Z_0}-1),m)}$ since~${q \leq m}$ by~\ref{item:manyhandles7} and~$\alpha$ is decreasing in its first argument. 
        Let~$\mathcal{P}''$ be a set of~${p-q}$ vertex-disjoint $W'$-handles containing exactly one element of~$\mathcal{P}_i$ for each~${i \in [p] \setminus [q]}$.
        Such a set~$\mathcal{P}''$ exists by~\ref{item:manyhandles1}. 
        For~${i \in [q]}$, let~${\mathcal{P}'_i := \mathcal{P}_i}$ and~${Z'_i := Z_i}$. 
        Observe that by~\ref{item:manyhandles8}, 
        for all distinct paths~$P$ and~$Q$ in~$\mathcal{P}''$, there exists ${j \in Z_0}$ such that~${\gamma_j(P) \neq \gamma_j(Q)}$.
        Thus, by Lemma~\ref{lem:ramsey}, there is a subset $\mathcal{P}'_{q+1}$ of $\mathcal{P}''$ 
        with~${\abs{\mathcal{P}'_{q+1}} = \alpha(q+1,\abs{Z_0}-1)}$
        such that for each~${i \in [m]}$, either 
        \begin{itemize}
            \item ${\gamma_i(P) = \gamma_i(Q)}$ for all~${P,Q \in \mathcal{P}'_{q+1}}$, or
            \item ${\gamma_i(P) \neq \gamma_i(Q)}$ for all distinct~${P,Q \in \mathcal{P}'_{q+1}}$,
        \end{itemize}
        and the second condition holds for some~${i \in Z_0}$. 
        Let 
        \[
            Z'_{q+1} := 
            \{ i \in Z_0 \colon \gamma_i(P) \neq \gamma_i(Q) \text{ for all distinct } P, Q \in \mathcal{P}'_{q+1} \}
            \text{ and } 
            Z'_0 := Z_0 \setminus Z'_{q+1}. 
        \]
        Let~${\mathfrak{P}' := (\mathcal{P}_i' \colon i \in [q+1])}$ and~${\mathcal{Z}' := (Z_i' \colon i \in \{0\} \cup [q+1])}$. 
        Then~${(W',\mathfrak{P}', \mathcal{Z}')}$ is a ${(c,q+1,q+1)}$-McGuffin, since~\ref{item:manyhandles2b} follows from the fact that~${\abs{\mathcal{P}'_{q+1}} \geq \alpha(q+1,\abs{Z_0}-1) \geq \alpha(q+1,\abs{Z_0'})}$, and the remaining conditions (\ref{item:manyhandles1}, \ref{item:manyhandles2}, \ref{item:manyhandles3}, \ref{item:manyhandles4}, \ref{item:manyhandles7}, and~\ref{item:manyhandles8}) are easy to check. 
        This contradicts the maximality of~${(q,p)}$.
        
        So we may assume that~${p < \hat{p}}$. 
        Let~$\Lambda$ be the subgroup of~$\Gamma$ consisting of all~${g \in \Gamma}$ 
        for which there is ${g' \in \gen{\bigcup_{i \in [p]} \gamma(\mathcal{P}_i)}}$ such that for all~${j \in Z_{0}}$ we have~${\pi_j(g) = \pi_j(g')}$. 
        Let~$\lambda$ be the induced ${\Gamma/\Lambda}$-labelling of~$G$. 
        Note that by the negation of~\ref{item:manyhandles5}, 
        neither ${\gen{\bigcup_{i \in [p]} \gamma(\mathcal{P}_i)}}$ nor~$\Lambda$ contains an element~$g$ such that~${\pi_j(g) \notin \Omega_j}$ for all~${j \in Z_0}$. 
        Therefore, 
        \begin{enumerate}
            [label=(\arabic*)]
            \item\label{eq:key} every cycle~$O$ of~$G$ for which~${\gamma_i(O) \notin \Omega_i}$ for all~${i \in [m]}$ is $\lambda$-non-zero. 
        \end{enumerate}
        
        Note that~$W'$ is a subwall of~$W$ of order~${c \geq u}$. 
        For any~${S \subseteq V(G)}$ of size at most~${u-1}$, there is a component~$X$ of~${G-S}$ containing a row of~$W'$. 
        By Lemma~\ref{lem:dominatedsubwall},~$\mathcal{T}_T$ dominates~$W'$, so the separation~${(V(G) \setminus V(X), S \cup V(X))}$ is in~$\mathcal{T}_T$, so~$X$ contains a vertex of~$\branch(W')$ and at least~${5t/6 - (u-1) > 5f_{m,\omega}(k)/6-(u-1) > 4u}$ vertices of~$T$. 
        By~\ref{eq:key}, every minimal subgraph~$H$ with~${\nu(H) \geq 1}$ is a $\lambda$-non-zero cycle. 
        Moreover, if~$H$ is a subgraph of~$G$ with~${\nu(H) < \nu(G) = k-1}$, then 
        by the induction hypothesis,~${\tau_\nu(H) \leq f_{m,w}(k-1) \leq 3u}$. 
        Hence, by Lemma~\ref{lem:cover}, we have that~$G$ contains a set of~$f_{\ref*{lem:addlinkage}} ( f_{\ref*{lem:ramsey}}(\alpha(1,m),m) )$ disjoint $\lambda$-non-zero~$\branch(W')$-paths. 
        We may assume that function~$w_{\ref*{lem:addlinkage}}$ is increasing in both of its arguments. 
        As~${\abs{Z_0} > 0}$ and~${p < \hat{p}}$, we have
        \begin{align*}
            c \geq \beta(p, \abs{Z_0}, \abs{Z}) 
            &\geq w_{\ref*{lem:addlinkage}}(\alpha(p+1,\abs{Z_0}),\beta(p+1,\abs{Z_0},\abs{Z})).
        \end{align*}
        
        Thus, by Lemma~\ref{lem:addlinkage} applied to~$W'$, there exist a $c'$-column-slice~$W''$ of~$W'$ 
        for some 
        \[
            c' \geq \beta(p+1,\abs{Z_0},\abs{Z}) \geq \beta(q+1,\abs{Z_0}-1,\abs{Z})
        \] 
        and a set~$\mathcal{P}'_i$ of~${f_{\ref*{lem:ramsey}}(\alpha(p+1,\abs{Z_0}),m)}$ vertex-disjoint $W''$-handles for each~${i \in [p+1]}$ such that 
        \begin{itemize}
            \item for each~${i \in [p]}$, the set~$\mathcal{P}'_i$ is a subset of the row-extension of~$\mathcal{P}_i$ to~$W''$ in~$W'$, 
            \item the paths in~${\bigcup_{i \in [p+1]} \mathcal{P}'_i}$ are vertex-disjoint, 
            \item the paths in~$\mathcal{P}'_{p+1}$ are $\lambda$-non-zero. 
        \end{itemize}

        By Lemma~\ref{lem:ramsey}, there exist a subset~${\mathcal{R}}$ of~${\mathcal{P}'_{p+1}}$ and a subset~$Z'$ of~$Z_0$ such that
        \begin{itemize}
            \item ${\abs{\gamma_j(\mathcal{R})} = \abs{\mathcal{R}}}$ for all~${j \in Z'}$,
            \item ${\abs{\gamma_j(\mathcal{R})} = 1}$ for all~${j \in Z_{0} \setminus Z'}$, and
            \item ${\abs{\mathcal{R}} \geq \alpha(p+1,\abs{Z_0}) \geq \alpha(q+1,\abs{Z_0}-1)}$.
        \end{itemize}
        Let~${p'' := p+1}$ if~$Z'$ is empty and let~${p'' := q+1}$ if~$Z'$ is non-empty, and 
        For~${i \in \{0\} \cup [p'']}$, let 
        \[
            Z''_i := 
            \begin{cases}
                Z_0 \setminus Z' & \text{if } i = 0,\\
                Z_i & \text{if } i \in [p''-1],\\
                Z' & \text{if } i = p''.\\
            \end{cases}
        \]
        For~${i \in [p''-1]}$, let~${\mathcal{P}_i'' := \mathcal{P}_i'}$ 
        and let~${\mathcal{P}_{p''} := \mathcal{R}}$. 

        We now show that~${\big( W'', (\mathcal{P}''_i \colon i \in [p'']), (Z''_i \colon i \in \{0\} \cup [p'']) \big)}$ is either a ${(c',q,p+1)}$-McGuffin or a ${(c', q+1, q+1)}$-McGuffin, contradicting the maximality of~${(q,p)}$. 
        First, observe that 
        since~$W$ is ${(\gamma',Z,\psi(\abs{Z}+1)+2)}$-clean, 
        every $N^W$-path is $\gamma_i$-zero for all~${i \in Z}$ and therefore
        if~$P'$ is the row-extension of~$P$ to~$W''$ in~$W'$, 
        then~${\gamma_i(P') = \gamma_i(P)}$ for all~${i \in Z}$, 
        implying~\ref{item:manyhandles3} and~\ref{item:manyhandles4} for~${i < p''}$. 
        By the definition of~$Z'$, properties~\ref{item:manyhandles3} and~\ref{item:manyhandles4} hold for~${i = p''}$.
        It remains to check~\ref{item:manyhandles8} when~$Z'$ is empty, ${q < i \leq p}$, and~${i' = p'' = p+1}$. 
        This is implied by the property that 
        the paths in~$\mathcal{P}'_{p+1}$ are $\lambda$-non-zero. 
    \end{proofofclaim}
    
    \begin{claim}
        There is a family~${(\mathcal{Q}_i \colon i \in [k])}$ of~$k$ disjoint subsets of~$\bigcup_{i \in [p]}\mathcal{P}_i$, 
        each of size at most $2^{\abs{Z_0}\omega}p$, 
        such that for each~${i \in [k]}$ and each~${j \in Z}$, 
        we have~${\gamma_j \big( \bigcup \mathcal{Q}_i \big) \notin \Omega_j}$.
    \end{claim}
    
    \begin{proofofclaim}
        Recursively, for each~${i \in [k]}$ we define~$\mathcal{Q}_i$ 
        containing at most~$2^{\abs{Z_0}\omega}$ elements of~$\mathcal{P}_j$ for all~${j \in [p]}$. 
        For each~${i \in [k]}$ and~${j \in [p]}$, let~${\mathcal{X}_{i,j} := \mathcal{P}_j \cap \bigcup_{i' \in [i-1]} \mathcal{Q}_{i'}}$, and note that we have~${\mathcal{X}_{1,j} = \emptyset}$ and~${\abs{\mathcal{X}_{i,j}} \leq (i-1) 2^{\abs{Z_0}w} \leq (k-1) 2^{mw}}$. 

        For each~${j \in [p]}$, select~${g_j \in \gamma(\mathcal{P}_j)}$ arbitrarily. 
        By Claim~\ref{clm:manyhandles}\ref{item:manyhandles4} and~\ref{item:manyhandles5}, 
        for each~${j \in [p]}$ there exists an integer~$c_j$ such that~${\pi_x \big( \sum_{j \in [p]} c_jg_j \big) \notin \Omega_x}$ for all~${x \in Z_0}$. 
        Hence, by Corollary~\ref{cor:omega-avoiding}, for each~${j \in [p]}$ there exists an integer~${d_j \in \big[ 2^{\abs{Z_0}\omega} \big]}$ 
        such that for all~${x \in Z_0}$, 
        we have~${\pi_x \big( \sum_{j \in [p]} d_jg_j \big) \notin \Omega_x}$. 
        Let~$I$ be the set of indices~${j \in [p]}$ such that~${Z_j \neq \emptyset}$. 
        Now~${\abs{\mathcal{P}_j} \geq \alpha(p,\abs{Z_0}) \geq \alpha(p,0) \geq k2^{mw} \geq \abs{\mathcal{X}_{i,j}} + d_j}$ for all~${j \in [p]}$ by Claim~\ref{clm:manyhandles}\ref{item:manyhandles2b}. 
        Hence, for each~${j \in [p]}$, we can select a set~$\mathcal{Y}_j$ of distinct $W'$-handles in~${\mathcal{P}_j \setminus \mathcal{X}_{i,j}}$ of size~$d_j -1$ if~${j \in I}$ and of size~$d_j$ otherwise. 
        By the definition of~$\mathcal{X}_{i,j}$ and Claim~\ref{clm:manyhandles}\ref{item:manyhandles2b}, 
        there are at least~${\alpha(p,\abs{Z_0}) - k2^{\abs{Z_0}\omega} \geq m\omega+1}$ distinct $W'$-handles in~${\mathcal{P}_j \setminus (\mathcal{X}_i \cup \mathcal{Y}_j)}$ for every~${j \in [p]}$. 
        Define 
        \[
            h := \sum_{j \in [p]} \sum_{P \in \mathcal{Y}_j} \gamma(P),
        \]
        and for~${j \in I}$ define~${S_j := \{ \gamma(P) \colon P \in \mathcal{P}_j \setminus (\mathcal{X}_{i,j} \cup \mathcal{Y}_j) \}}$. 
        By Claim~\ref{clm:manyhandles}\ref{item:manyhandles3}, for all~${j \in I}$, for all distinct~$g$,~$g'$ in~$S_j$, and for all~${x \in Z_j}$, 
        we have~${\pi_x(g) \neq \pi_x(g')}$
        and so~${\abs{S_j} > mw}$.
        Now by Lemma~\ref{lem:vectorsum}, 
        there is a family~${(Q_j \colon j \in I)}$ of paths 
        such that~${Q_j \in \mathcal{P}_j \setminus (\mathcal{X}_{i,j} \cup \mathcal{Y}_j)}$ 
        for all~${j \in I}$, 
        and~${\pi_x \big( h + \sum_{j \in I} \gamma(Q_j) \big) \notin \Omega_x}$ for all~${x \in Z}$. 
        Hence, let~${\mathcal{Q}_i := \{Q_j \colon j \in I \} \cup \bigcup_{j \in [p]} \mathcal{Y}_j}$, 
        and note that~${\abs{\mathcal{Q}_i} \leq \sum_{j \in [p]} d_j \leq 2^{\abs{Z_0}w}p}$.
    \end{proofofclaim}

    We now complete the proof of the theorem. 
    Since~$W'$ has at least~${\beta(1,0,\abs{Z})}$ columns, 
    there is a set~${\{W_i \colon i \in [k]\}}$ of~$k$ vertex-disjoint 
    ${c_{\ref*{lem:omega-avoiding-cycle}}(2^{m\omega}\hat{p},\psi(\abs{Z}+1)+2, m, \omega)}$-column-slices of~$W'$. 
    Note that  the number of rows of~$W'$ is at least~${\psi(\abs{Z}) \geq r_{\ref*{lem:omega-avoiding-cycle}}(2^{m\omega}\hat{p},\psi(\abs{Z}+1)+2, m, \omega)}$. 
    For each~${i \in [k]}$, let~$\mathcal{Q}^\ast_i$ be the row-extension of~$\mathcal{Q}_i$ to~$W_i$. 
    By Lemma~\ref{lem:omega-avoiding-cycle},
    for each~${i \in [k]}$ there is a cycle~$O_i$ in~${W_i \cup \bigcup \mathcal{Q}^\ast_i}$ with~${\gamma_j(O_i) \notin \Omega_j}$ for all~$j\in [m]$. 
    Observe that for every vertex~${v \in V(G)}$, there are at most two indices~${i \in [k]}$ such that~${v \in V(W_i \cup \bigcup \mathcal{Q}^\ast_i)}$. 
    Hence,~${\nu(G) \geq k}$, a contradiction. 
\end{proof}

\section{Conclusion}
\label{sec:conclusion}

In this work, we proved that a half-integral analogue of the Erd\H{o}s-P\'{o}sa theorem holds for 
cycles in graphs labelled with a bounded number of abelian groups, whose values avoid a bounded number of elements of each group. 
We conclude with some open problems. 

In the proof of our theorem, the theorem of Wollan~\cite{Wollan2010} about $\gamma$-non-zero $A$-paths was important. 
This theorem implies that an analogue of the Erd\H{o}s-P\'{o}sa theorem holds for the odd $A$-paths, and for $A$-paths intersecting a prescribed set of vertices.
Bruhn, Heinlein, and Joos~\cite{BruhnMJ2018} further showed that an analogue of the Erd\H{o}s-P\'{o}sa theorem holds for $A$-paths of length at least~$\ell$, and for $A$-paths of even length,
but interestingly, they also showed that for every composite integer~${m > 4}$ and every~${d \in \{0\} \cup [m-1]}$,
no such analogue holds for $A$-paths of length~$d$ modulo~$m$. 
Later, Thomas and Yoo~\cite{YooT20202} characterised the abelian groups~$\Gamma$ and elements~${\ell \in \Gamma}$ where an analogue of the Erd\H{o}s-P\'{o}sa theorem holds for $A$-paths of $\gamma$-value~$\ell$. 
We would like to ask whether a statement similar to Theorem~\ref{thm:main} holds for $A$-paths.

\begin{question}
    \label{ques1}
    For every pair of positive integers~$m$ and~$\omega$, does there exist a function~${g_{m,\omega} \colon \mathbb{N} \to \mathbb{N}}$ satisfying the following property?
    \begin{itemize}
        \item For each~${i \in [m]}$, let~$\Gamma_i$ be an abelian group and let~$\Omega_i$ be a subset of~$\Gamma_i$. 
        Let~$G$ be a graph, let~$A$ be a set of vertices in~$G$, and for each~${i \in [m]}$, 
        let~${\gamma_i}$ be a~${\Gamma_i}$-labelling of~$G$, 
        and let~$\mathcal{P}$ be the set of all $A$-paths of~$G$ whose $\gamma_i$-value is in~${\Gamma_i \setminus \Omega_i}$ for all~${i \in [m]}$. 
        If ${\abs{\Omega_i} \leq \omega}$ for all~${i \in [m]}$, then there exists 
        either a half-integral packing of~$k$ paths in~$\mathcal{P}$, 
        or a hitting set for~$\mathcal{P}$ of size at most~${g_{m,\omega}(k)}$. 
   \end{itemize}
\end{question}

Our next question relates to directed labellings of graphs.
Let~$\Gamma$ be a group (not necessarily abelian). 
A \emph{directed $\Gamma$-labelling} of a graph~$G$ is a function~${\gamma}$ from the set of oriented edges ${\vec{E}(G) := \{ (e, w) \colon e = uv \in E(G), w \in \{u,v\} \}}$ to~$\Gamma$ such that~${\gamma(e,u) = -\gamma(e,v)}$ for each edge~${e = uv}$. 
Given a walk~${W := v_t e_1 v_1 e_2 \cdots e_t v_t}$, we define~${\gamma(W) := \sum_{j = 1}^t \gamma(e_j, v_{j})}$, and say that~$W$ \emph{corresponds} to a cycle~$O$ if~${E(O) = E(W)}$ and~$v_t$ is the only repeated vertex of~$W$. 
It is straightforward to check that if any walk corresponding to a cycle~$O$ has value~$0$, then all walks corresponding to~$O$ do, so it makes sense to consider non-zero cycles with respect to a directed labelling. 
Note that if~$\Gamma$ is abelian and~$W_1$ and~$W_2$ are walks corresponding to the same cycle~$O$, then~${\gamma(O_1) = \pm \gamma(O_2)}$. 
If~$\Gamma$ is not abelian, then the choice of start vertex for the corresponding walk does matter as well. 
Hence, in this case, we are really considering cycles together with a specified start vertex and direction. 

Huynh, Joos, and Wollan~\cite{HuynhJW2017} conjectured that a half-integral analogue of the Erd\H{o}s-P\'{o}sa theorem holds for cycles which are non-zero with respect to a fixed number of directed labellings. 
We ask whether a statement similar to Theorem~\ref{thm:main} holds for directed labellings. 
If it does, then it would imply the conjecture of Huynh, Joos, and Wollan. 

\begin{question}
    \label{ques2}
    For every pair of positive integers~$m$ and~$\omega$, does there exist a function~${g_{m,\omega} \colon \mathbb{N} \to \mathbb{N}}$ satisfying the following property?
    \begin{itemize}
        \item For each~${i \in [m]}$, let $\Gamma_i$ be a group and let~$\Omega_i$ be a subset of~$\Gamma_i$. 
        Let~$G$ be a graph and for each~${i \in [m]}$, let~${\gamma_i}$ be a directed~${\Gamma_i}$-labelling of~$G$, and let~${\mathcal{O}}$ be the set of all cycles of~$G$ which have a corresponding walk $W$ such that $\gamma_i(W)$ is in~${\Gamma_i \setminus \Omega_i}$ for all~${i \in [m]}$. 
        If~${\abs{\Omega_i} \leq \omega}$ for all~${i \in [m]}$, then there exists 
        either a half-integral packing of~$k$ cycles in~$\mathcal{O}$, 
        or a hitting set for~$\mathcal{O}$ of size at most~${g_{m,\omega}(k)}$. 
    \end{itemize}
\end{question}

As discussed in Section~\ref{subsec:obstructionsEP}, an analogue of the Erd\H{o}s-P\'{o}sa theorem does not hold for the cycles described in Theorem~\ref{thm:main}.
It was shown that in graphs of sufficiently high connectivity~\cite{Thomassen2001, RautenbackR2001, KawarabayashiR2009, Joos2017}, 
an analogue of the Erd\H{o}s-P\'{o}sa theorem holds for odd cycles. 
We ask whether a similar phenomenon happens for the cycles described in Theorem~\ref{thm:main}. 

\begin{question}
    \label{ques3}
    For every pair of positive integers~$m$ and~$\omega$, do there exist functions~${g_{m,\omega} \colon \mathbb{N} \to \mathbb{N}}$ and~${g'_{m,\omega} \colon \mathbb{N} \to \mathbb{N}}$ satisfying the following property? 
    \begin{itemize}
        \item For each~${i \in [m]}$, let~$\Gamma_i$ be an abelian group and let~$\Omega_i$ be a subset of~$\Gamma_i$. 
        Let~$G$ be a graph, and for each~${i \in [m]}$, let~${\gamma_i}$ be a~${\Gamma_i}$-labelling of~$G$, 
        and let~${\mathcal{O}}$ be the set of all cycles of~$G$ whose $\gamma_i$-value is in~${\Gamma_i \setminus \Omega_i}$ for all~${i \in [m]}$.
        If~$G$ is ${g'_{m, \omega}(k)}$-connected and~${\abs{\Omega_i} \leq \omega}$ for all~${i \in [m]}$, then there exists 
        either a set of~$k$ vertex-disjoint cycles in~$\mathcal{O}$, 
        or a hitting set for~$\mathcal{O}$ of size at most~${g_{m,\omega}(k)}$. 
    \end{itemize}
\end{question}

Similar to Question~\ref{ques3}, we ask whether an analogue of the Erd\H{o}s-P\'{o}sa theorem holds for the $A$-paths described in Question~\ref{ques1}, 
and for the cycles described in Question~\ref{ques2}, 
in the case of highly connected graphs. 

\bibliographystyle{abbrv}
\bibliography{main}

\end{document}